\documentclass[12pt,reqno]{amsart}
\usepackage{amsthm,amsmath,amssymb,amscd,graphicx,enumerate,xcolor,subsupscripts,fullpage}
\usepackage[all]{xy}
\usepackage[latin1]{inputenc}
\usepackage[T1]{fontenc}
\usepackage{hyperref}
\usepackage{ifthen}

\newboolean{workmode}
\setboolean{workmode}{false}

\input xy
\xyoption{all}

\makeatletter
\def\th@plain{%
  \thm@notefont{}
  \itshape 
}
\def\th@definition{%
  \thm@notefont{}
  \normalfont 
}
\makeatother

\newcommand{\op}[1]{{#1}^{\sf op}}

\newcommand{\R}{\mathbb{R}}

\newcommand{\set}{{\sf Set}}

\newcommand{\vect}{{\sf Vect}}
\newcommand{\ab}{{\sf Ab}}
\newcommand{\grp}{{\sf Group}}
\newcommand{\euc}{{\sf Euc}}
\newcommand{\open}{{\sf Open}}

\newcommand{\rra}{\rightrightarrows}

\newcommand{\ep}{\mathfrak{e}}

\newcommand{\Cc}{\mathcal{C}}
\newcommand{\Cd}{\mathcal{D}}

\newcommand{\Cn}{\mathcal{N}}

\newcommand{\calA}{{\mathcal{A}}}

\newcommand{\calC}{{\mathcal{C}}}
\newcommand{\calD}{{\mathcal{D}}}

\newcommand{\calF}{{\mathcal{F}}}
\newcommand{\calG}{{\mathcal{G}}}
\newcommand{\calH}{{\mathcal{H}}}

\newcommand{\calM}{{\mathcal{M}}}
\newcommand{\calN}{{\mathcal{N}}}
\newcommand{\calO}{{\mathcal{O}}}

\newcommand{\calQ}{{\mathcal{Q}}}
\newcommand{\calR}{{\mathcal{R}}}
\newcommand{\calS}{{\mathcal{S}}}
\newcommand{\calT}{{\mathcal{T}}}
\newcommand{\calU}{{\mathcal{U}}}
\newcommand{\calV}{{\mathcal{V}}}

\newcommand{\QQ}{{\mathbb{Q}}}  
\newcommand{\RR}{{\mathbb{R}}}  
\newcommand{\TT}{{\mathbb{T}}}  
\newcommand{\ZZ}{{\mathbb{Z}}}  

\newcommand{\act}{{\operatorname{act}}} 
\newcommand{\Diff}{{\operatorname{Diff}}}  
\newcommand{\diffeol}{\mathbf{Diffeol}} 
\newcommand{\ev}{{\operatorname{ev}}}   
\newcommand{\Hom}{{\operatorname{Hom}}} 
\newcommand{\id}{{\operatorname{id}}}  
\newcommand{\IZ}{{\operatorname{IZ}}} 
\newcommand{\im}{{\operatorname{im}}}  
\newcommand{\pr}{{\operatorname{pr}}} 
\newcommand{\supp}{{\operatorname{supp}}}  

\newcommand{\CIN}{{C^\infty}}   
\newcommand{\del}{\partial}  
\newcommand{\ftimes}[2]{{\lrsubscripts{\times}{#1}{#2}}} 
\newcommand{\toto}{{~\rightrightarrows~}} 

\newcommand{\ifwork}[1]{\ifthenelse{\boolean{workmode}}{#1}{}}
\newcommand{\comment}[1]{}
\newcommand{\mute}[1]{}
\newcommand{\printname}[1]{}

\ifwork{
\renewcommand{\comment}[1]{{\marginpar{*}\ \scriptsize{#1}\ }}
\renewcommand{\printname}[1]
    {{\color{brown}{\makebox[0pt]{\hspace{-1.0in}\raisebox{8pt}{\tiny #1}}}}}
}

\newcommand{\labell}[1] {\label{#1} \printname{#1}}

\makeatletter
\newtheorem*{rep@theorem}{\rep@title}
\newcommand{\newreptheorem}[2]{%
\newenvironment{rep#1}[1]{%
 \def\rep@title{#2 \ref{##1}}%
 \begin{rep@theorem}}%
 {\end{rep@theorem}}}
\makeatother

\theoremstyle{plain}
\newtheorem{theorem}{Theorem}[section]
\newreptheorem{theorem}{Theorem} 	

\newtheorem{proposition}[theorem]{Proposition}

\newtheorem{corollary}[theorem]{Corollary}
\newtheorem{lemma}[theorem]{Lemma}

\theoremstyle{definition}
\newtheorem{definition}[theorem]{Definition}
\newtheorem{example}[theorem]{Example}
\newtheorem{examples}[theorem]{Examples}
\newtheorem{remark}[theorem]{Remark}

\newtheorem{construction}[theorem]{Construction}

\definecolor{jaw}{rgb}{0,.5,0}

\newtheorem*{corollary*}{Corollary}

\definecolor{jaw}{HTML}{009933}
\newcommand{\snote}[1]{{\comment{{\color{red}{\em S:} #1}}}} 
\newcommand{\jnote}[1]{{\comment{{\color{jaw}{\em J:} #1}}}} 

\title{Sheaves, principal bundles, and \v{C}ech cohomology for diffeological spaces}

\author{Derek Krepski}
\address{Department of Mathematics, University of Manitoba, Winnipeg, Manitoba, CA}
\email{derek.krepski@umanitoba.ca}

\author{Jordan Watts}
\address{Department of Mathematics, Central Michigan University, Mount Pleasant, Michigan, USA}
\email{jordan.watts@cmich.edu}

\author{Seth Wolbert}
\address{Department of Mathematics and Statistics, Minnesota State University, Mankato, Minnesota, USA}
\email{seth.wolbert@mnsu.edu}

\date{\today}

\begin{document}

\begin{abstract}The purpose of this note is to define sheaves for diffeological spaces and give a construction of their \v{C}ech cohomology.  As an application, we prove that the first degree \v{C}ech cohomology classes for the sheaf of smooth functions to an abelian diffeological group $G$ classify diffeological principal $G$-bundles.
\end{abstract}

\maketitle

\section{Introduction}\labell{s:introduction}

Diffeological spaces yield a convenient category on which one may study singular spaces and function spaces in the context of differential geometry using a simple yet powerful language.  These spaces were defined by Souriau in the 1980s \cite{Souriau}, and is a variant of previously defined categories, in particular Chen spaces, developed earlier by K.-T. Chen in the 1970s \cite{KTChen}.  An interesting feature of these spaces is the large number of constructions of differentiable manifolds or topological spaces that carry over to diffeological spaces.  These constructions range from broader categorical type constructions (\emph{e.g.}\ fiber products, coproducts, quotients, etc.) to more specific adaptations of higher level concepts (including de Rham forms and their cohomology, principal/fiber bundles, and classifying spaces; see \cite{ChristensenWu} and \cite{MagnotWatts} for two independent developments of classifying spaces).  The purpose of this note is to build a theory of sheaves over diffeological spaces and their \v{C}ech cohomology.

The main idea in defining a sheaf over a diffeological space is to use plots instead of open sets as in the classical setting of topology.  With this framework we can define \v{C}ech cohomology, and then prove:

\begin{reptheorem}{t:main}
Let $X$ be a diffeological space and $G$ an abelian diffeological group.  There is a natural bijection between the isomorphism classes of principal $G$-bundles on $X$ and the group of \v{C}ech cohomology classes $\check{H}^1(X,G)$.
\end{reptheorem}

As an immediate corollary, we combine the above theorem with the main results of \cite{MagnotWatts,ChristensenWu}.  Given a diffeological space $X$ and a diffeological group $G$, there is a natural bijection from isomorphism classes of so-called D-numerable principal $G$-bundles on $X$ to smooth homotopy classes of smooth maps from $X$ to the classifying space $BG$ of $G$; denote the set of homotopy classes by $[X,BG]$.  (Note: the Hausdorff, second-countable, and smoothly paracompact requirements appearing in \cite{MagnotWatts} are unnecessary; see \cite[Lemmas 4.2, 4.3]{ChristensenWu}.)  See the cited papers for definitions and full details.

\begin{corollary*}\labell{c:main}
Let $X$ be a diffeological space and $G$ an abelian diffeological group such that all principal $G$-bundles are D-numerable (\emph{e.g.}\ a manifold).  There is a natural bijection between $\check{H}^1(X,G)$ and $[X,BG]$.
\end{corollary*}

%
This paper was largely inspired by the theory of de Rham cohomology developed for diffeological spaces as well as the work of Behrend and Xu in \cite{BehrendXu}.  Indeed, the sheaf of de Rham forms over a diffeological space provides a prototypical sheaf of vector spaces.  While in this paper we stick to the more classical geometric language in which most diffeological literature is cast, a lot of insight can be gained through using the language of Behrend and Xu; namely, Grothendieck topologies and stacks.  In particular, as shown in \cite{WattsWolbert}, any diffeological space $X$ has an associated stack $\mathcal{X}$ via the Grothendieck construction, and the sheaves and global sections defined in this paper are sheaves and global sections (in the stacky sense) over $\mathcal{X}$.  

Similar work on a diffeological context for sheaves already appears in the work of Nezhad and Ahmadi \cite{NezhadAhmadi}, in which the authors also define sheaves and presheaves of sets in the same way we do, but they develop the theory in a different direction than the present paper, defining so-called quasi-sheaves which are presheaves satisfying nice properties with respect to generating families and limits taken over them.  It is possible that some of the theory developed in our paper may be generalized to quasi-sheaves.  The authors in \cite{NezhadAhmadi} also connect their theory to the D-topology of a diffeological space, establishing the desired fact that a sheaf on a diffeological space using the plot definition yields a sheaf using the topological definition.

Patrick Iglesias-Zemmour studies an independent and different version of \v{C}ech cohomology in \cite{PIZCech}, with the purpose of calculating the \v{C}ech cohomology of the sheaf of locally constant functions to $\R$ and investigating the obstruction to extending the de Rham isomorphism of sheaf cohomology to the context of diffeological spaces.  We compare our definition with his in Subsection~\ref{ss:iz cohom}.

This paper is the first in many pushing the theory of sheaves into the diffeological realm.  Future work includes developing a satisfactory theory of stalks, spectral sequence techniques, non-abelian \v{C}ech cohomology, and the connections to the theory of diffeological groupoids (see Remark~\ref{r:groupoidrep} and \cite{vdS,Watts-diffeol-gpds}).

The sections of the paper are organized as follows: in Section~\ref{s:diffeology}, we give a brief overview of diffeological spaces along with various constructions of them that will be used throughout the paper.  In Section~\ref{s:sheaves}, we give a definition and some basic properties of sheaves over diffeological spaces.  In Section~\ref{s:cech}, we define \v{C}ech cohomology for sheaves over a diffeological space.  In Section~\ref{s:bundles}, we prove the main theorem of the paper: that principal $G$-bundles of an abelian diffeological group $G$ are classified by the degree 1 \v{C}ech cohomology of the sheaf of smooth $G$-valued functions.  We also show that there is a (non-natural) bijection between $\check{H}^1(X,G)$ and certain groupoid actions on the nebula of a sufficiently ``good'' generating family of $X$; see Corollary~\ref{c:contr gen fam}.  Finally, in Section~\ref{s:apps}, we give applications and some examples: we recover the classical picture in the case that $X$ is a smooth manifold and $G$ a Lie group, we look at the case of $X$ an irrational torus (an example not captured by the classical topological theory), and consider a relationship between group cohomology and principal $\RR$-bundles, as well as a potential generalization of the notion of ``fineness'' for a sheaf.  We also reproduce the classical result that \v{C}ech cohomology is the right derived functor of the global sections functor. This is combined with the results of the earlier subsections to obtain for a general irrational torus $T=\RR/\Gamma$ (where $\Gamma$ is a diffeologically discrete subgroup of $\RR$), the principal $T$-bundles over an orbit space of a proper Lie groupoid $X$ are classified by $\check{H}^2(X,\Gamma)$ (see Corollary~\ref{c:derived functor}).  This section ends with the aforementioned comparison with \cite{PIZCech}.

\section{Diffeological spaces}\labell{s:diffeology}

We start with the diffeological preliminaries required for this paper.

\begin{definition}[Diffeology]\labell{d:diffeology}
Let $X$ be a set.  A {\em parametrization} of $X$ is a
map of sets $p \colon U \to X$ where $U$ is an open subset of Euclidean space (no fixed dimension).  A {\em diffeology} $\calD_X$ on $X$ is a set of parametrizations satisfying the following three axioms:
\begin{enumerate}
\item {\em (Covering Axiom)} $\calD_X$ contains all constant maps into $X$.
\item {\em (Locality Axiom)} Given a parametrization $p\colon U\to X$, if there is an open cover $\{U_\alpha\}$ of $U$ and a family $\{p_\alpha\colon U_\alpha\to X\}\subseteq\calD_X$ such that $p|_{U_\alpha}=p_\alpha$ for each $\alpha$, then $p\in\calD_X$.
\item {\em (Smooth Compatibility Axiom)} Given $p\colon U\to X$ in $\calD_X$, an open subset $V$ of a Euclidean space, and a smooth map $f\colon V\to U$, the composition $p\circ f$ is in $\calD_X$.
\end{enumerate}
A set $X$ equipped with a diffeology $\calD_X$ is called a
{\em diffeological space}, and the parametrizations $p\in\calD_X$ {\em plots} of $X$.  In particular, there is a unique {\em empty plot} $\ep\colon\emptyset\to X$.  We will often denote the domain of a plot $p$ by $U_p$.

A map $\varphi\colon X\to Y$ between diffeological spaces $X$ and $Y$ is {\em (diffeologically) smooth} if $\varphi \circ p \in \mathcal{D}_Y$ for any plot $p \in \mathcal{D}_X$. If $\varphi$ is furthermore a bijection with smooth inverse, then it is a \emph{diffeomorphism}.  Denote the set of all smooth maps $X\to Y$ by $\CIN(X,Y)$.
\end{definition}

A first example of a diffeological space is a smooth manifold, equipped with the diffeology of all $C^\infty$ maps from Euclidean open sets into it.  In fact, smooth manifolds constitute a full subcategory of the category $\diffeol$ of diffeological spaces with smooth maps between them, and so diffeology is a generalization of the theory of smooth manifolds.  But that's not all: unlike smooth manifolds, $\diffeol$ is extremely versatile, being a complete, cocomplete quasi-topos \cite{BaezHoffnung}.  In particular, there are natural diffeological structures on subsets, quotients, products, sums, and function spaces defined as follows (see \cite{PIZ} for proofs that these satisfy the axioms of a diffeological space). 

\begin{definition}[Basic Categorical Constructions]\labell{d:categ constr}
Let $X$ be a diffeological space and $\{X_\alpha\}$ a family of diffeological spaces.
\begin{itemize}
\item Given a subset $Y\subset X$, the \emph{subset diffeology} on $Y$ is the collection of all plots of $X$ whose images lie in $Y$.
\item Given an equivalence relation $\sim$ on $X$ with quotient map $\pi$, the \emph{quotient diffeology} on $X/\!\sim$ is the collection of all parametrizations $p$ in which for each $u\in U_p$ there is an open neighborhood $V$ of $u$ and a plot $q\colon V\to X$ such that $p|_V=\pi\circ q$.
\item The \emph{product diffeology} on the set $\prod_\alpha X_\alpha$ is the collection of all parametrizations $p$ for which $\pr_\alpha\circ p$ is a plot of $X_\alpha$ for each $\alpha$, where $\pr_\alpha$ is the projection map.
\item The \emph{sum diffeology} on the set $\coprod_\alpha X_\alpha$ is the collection of all parametrizations $p$ in which for each $u\in U_p$ there is an open neighborhood $V$ of $u$ and a plot $q\colon V\to X_\alpha$ for some $\alpha$ such that $p|_V=q\circ i_\alpha$, where $i_\alpha$ is the inclusion map.
\item Given another diffeological space $Y$, the \emph{(standard) functional diffeology} on $\CIN(X,Y)$ is the collection of all parametrizations $p\colon U\to\CIN(X,Y)$ so that for all $(u,x)\in U\times X$, the assignment $(u,x)\mapsto p(u)(x)$ is smooth.
\end{itemize}
\end{definition}

Using the definitions above, one can construct limits and colimits, such as fiber products.

\begin{definition}[Fibered Product]\labell{d:fiber product}
Given a family of smooth maps $\varphi_\alpha\colon X_\alpha\to Z$, define the \emph{fibered product}  of $\{\varphi_\alpha\}$ to be the set 
$$\left\{(x_\alpha)\in \prod_\alpha X_\alpha\;\Big|\; \varphi_{\alpha_1}(x_{\alpha_1})=\varphi_{\alpha_2}(x_{\alpha_2})~\forall \alpha_1,\alpha_2\right\}$$
equipped with the subset diffeology induced by the product diffeology on $\prod_\alpha X_\alpha$.  This space is the unique diffeological space, up to diffeomorphism, satisfying the standard universal property of a fibered product.
\end{definition}

Just as one might use topological bases to generate a topology, diffeological spaces enjoy generating families.

\begin{definition}[Generating Families and Nebulae]\labell{d:nebula}
A {\em (covering) generating family} of a diffeological space $X$ is a family $\calQ\subseteq\calD_X$ such that for any plot $p\in\calD_X$ there exist an open cover $\{U_\alpha\}$ of $U_p$, a family $\{q_\alpha\colon V_\alpha\to X\}$ in $\calQ$, and a family of smooth maps $\{f_\alpha\colon U_\alpha\to V_\alpha\}$ such that $p|_{U_\alpha}=q_\alpha\circ f_\alpha$ for each $\alpha$.

The {\em nebula} of a generating family $\calQ$ is the diffeological sum $\Cn(\calQ):=\coprod_{q \in \calQ} U_q$. The sum includes a copy of the domain of each plot from $\calQ$; hence, multiple copies of the same Euclidean open set may appear in the sum.  The nebula comes equipped with a (smooth) \emph{evaluation map} $\ev\colon\calN(\calQ)\to X$ sending $u\in U_q$ to $q(x)$.
\end{definition}

Especially when discussing principal bundles later in this paper, we will need special classes of smooth injections and surjections between diffeological spaces.

\begin{definition}[Inductions \& Subductions]\labell{d:induction-subduction}
A smooth map $\varphi\colon X\to Y$ is an {\em induction} if it is injective, and a parametrisation $p$ of $X$ is a plot if $\varphi\circ p$ is a plot of $Y$.  Similarly, a smooth map $\varphi\colon X\to Y$ is a {\em subduction} if it is surjective, and for any plot $p\in\calD_Y$ there exist an open cover $\{U_\alpha\}$ of $U_p$ and plots $\{q_\alpha\colon U_\alpha\to X\}$ such that $p|_{U_\alpha}=\varphi\circ q_\alpha$ for each $\alpha$; call each $q_\alpha$ a \emph{local lift} of $p$ over $U_\alpha$.
\end{definition}

Examples of inductions include inclusion maps of subsets equipped with the subset diffeology, and examples of subductions include quotient maps to quotients equipped with the quotient diffeology.  In particular, the evaluation map of a generating family $\calQ$ of a diffeological space $X$ is a subduction, as $X$ is diffeomorphic to the quotient of $\calN(\calQ)$ by the equivalence relation $u_1\sim u_2$ if $\ev(u_1)=\ev(u_2)$.

\begin{example}[Generating Family Induced by a Subduction]\labell{x:nebula}
Let $\varphi\colon X\to Y$ be a subduction.  Let $p\colon U\to Y$ be a plot.  There exist an open cover $\{U_\alpha\}$ of $U$ and plots $q^p_\alpha\colon U_\alpha\to X$ such that $p|_{U_\alpha}=\varphi\circ q^p_\alpha$ for each $\alpha$.  The collection of all $\varphi\circ q^p_\alpha$ as $p$ runs over $\calD_Y$ and $\alpha$ over all open sets in a suitable open cover of $U_p$ is a generating family $\calQ$ of $Y$ with evaluation map $\ev$.  Moreover, there is a smooth map $\widetilde{\ev}\colon\calN(\calQ)\to X$, equal to $q^p_\alpha$ on each plot domain $U_{\varphi\circ q^p_\alpha}$, satisfying $\varphi\circ\widetilde{\ev}=\ev$; hence $\widetilde{\ev}$ is a lift of $\ev$ to $X$.
\end{example}

Finally, we define diffeological groups.

\begin{definition}\labell{d:diffeol gp}
A \emph{diffeological group} $G$ is a group equipped with a diffeology such that multiplication and inversion are smooth.
\end{definition}

\begin{example}\labell{x:diffeol gp}
Any Lie group is a diffeological group.  The irrational torus $\RR/(\ZZ+\alpha\ZZ)$ is a diffeological group, but not a Lie group.  Here $\alpha$ is an irrational number, and $\ZZ+\alpha\ZZ$ the diffeologically discrete subgroup of $\RR$.  The group of diffeomorphisms $\Diff(X)$ of a diffeological space is a diffeological group when equipped with an appropriate functional diffeology (see \cite{PIZ} for details).
\end{example}

\begin{example}\labell{x:functions to gp}
Let $X$ be a diffeological space and $G$ a diffeological group.  The functional diffeology on $\CIN(X,G)$ admits a smooth pointwise multiplication and inversion map, making it into a diffeological group.
\end{example}

\section{Sheaves over a diffeology}\labell{s:sheaves}

In this section we introduce presheaves and sheaves of a diffeological space, their global sections, morphisms, and prove that our definitions behave well under pullback.    The novelty of our approach is that we replace the category of open sets of a topological space with inclusions as arrows, on which a classical sheaf is defined, with the category of plots (defined below) of a diffeological space.  While this allows us to obtain results not available in the topological category, any sheaf on a smooth manifold according to our definition below yields a classical sheaf on the open sets of the manifold.  We end the section with a result familiar from classical sheaf theory: that taking global sections is a right exact functor (see Proposition~\ref{p:ses}).

\begin{definition}[Category of Plots]\labell{d:categ plots}
Given a diffeological space $X$, {\em the category of plots} $\calD(X)$ is the category comprising plots of $X$ as objects and commutative triangles
\[\xymatrix@R=0.5em{U \ar[dr]^{p} \ar[dd]_f & \\ & X \\ V\ar[ur]_q & }\] 
as morphisms, where $p$ and $q$ are plots and $f\colon U\to V$ is smooth.  For brevity, designate such a commutative triangle as $f_p^q\colon p\to q$. 
\end{definition}

\begin{definition}[$\calC$-Valued Presheaf]\labell{d:presheaf}
Fix a diffeological space $X$.  Let $\Cc$ be the category of sets $\set$, groups $\grp$, abelian groups $\ab$, or vector spaces $\vect$.  A {\em $\Cc$-valued presheaf over $X$} is a functor $F\colon\op{\calD(X)} \to \Cc$.  A {\em global section} of a $\calC$-valued presheaf $F\colon\op{\calD(X)}\to\calC$ is a family of objects $\{\eta_p\in F(p)\mid p\in\calD_X\}$ which is natural with respect to arrows of $\calD(X)$; \emph{i.e.}\ $F(f_p^q)(\eta_q)=\eta_p$ for every arrow $f_p^q\in\calD(X)$; we sometimes refer to such a family as a \emph{coherent family}. Denote the set of global sections of $F$ by $F(X)$.
\end{definition}

Global sections of a $\calC$-valued presheaf $F$ form an object in $\calC$ via plotwise-defined operations, emulating what one would expect from the classical case.  For instance, if $\calC$ is $\grp$, and $\{\eta_p\}$ and $\{\mu_p\}$ are two global sections of $F$, then $\{\eta_p\}\{\mu_p\}=\{\eta_p\mu_p\}$.

While the value $F(p)$ is determined by the plot $p$, we are often interested in the case where the value of $F$ is determined only by the domain of $p$.

\begin{definition}[Domain-Determined]\labell{d:domain-determined}
A $\Cc$-valued presheaf $F\colon \op{\Cd(X)}\to \Cc$ is {\em domain-determined} if $F(p)=F(q)$ for any two plots $p,q$ with $U_p=U_q$; and $F(f_p^q)$ is uniquely determined by $f$, for any smooth arrow $f_p^q$ of $\calD(X)$.
\end{definition}

Given a topological space $Z$, let $\open(Z)$ be the category of open subsets of $Z$ with inclusions as arrows.  For a presheaf $F\colon\op{\Cd(X)}\to \Cc$ and plot $p\colon U\to X$, an inclusion of an open subset $\iota\colon V\to U$ induces a map $\iota_{p\circ \iota}^p\colon p|_V\to p$, from which it follows that for each plot $p\colon U\to X$, one obtains a presheaf in the classical sense
$F|_p\colon\op{\open(U)}\to \Cc$ sending $V\subseteq U$ to $F(p\circ \iota)$.

\begin{definition}[$\calC$-Valued Sheaf]\labell{d:sheaf}
A {\em $\calC$-valued sheaf} $F\colon\op{\Cd(X)}\to \Cc$ over a diffeological space $X$ is a $\calC$-valued presheaf that sends the empty plot $\ep$ to the terminal object of $\calC$, and whose restriction $F|_p:\op{\open(U)}\to \Cc$ is a sheaf for every plot $p\colon U\to X$.
\end{definition}

\begin{remark}\labell{r:coprods}
Our requirement that $F$ takes $\ep$ to the terminal object allows us to conclude that any plot $p\colon U\amalg V \to X$ satisfies $F(p)=F(p|_U)\times F(p|_V)$.  It follows that a sheaf $F$ on $X$ is completely determined by its values on plots with connected domains, and so it is generally enough to consider such plots.
\end{remark}

Any $\calC$-valued sheaf on a smooth manifold $M$ yields a sheaf in the classical sense.  For certain sheaves such as $C^\infty(\,\cdot\,,Y)$ below, the converse is immediate.

\begin{examples}\labell{x:sheaves}
Let $X$ and $Y$ be diffeological spaces. 
Three important examples of sheaves over diffeological spaces include:
\begin{enumerate}
\item For each degree $k\geq0$, denote by $\Omega^k\colon\op{\calD(X)}\to \vect$ the domain-determined sheaf given by $\Omega^k(p):=\Omega^k(U_p)$ and $\Omega^k(f_p^q)(\eta):=f^*\eta$.  The global sections of this sheaf $\Omega^k(X)$ are precisely the de Rham $k$-forms on $X$.

\item Denote by $C^\infty(\,\cdot\,, Y)\colon\op{\calD(X)}\to \set$ the domain-determined sheaf given by $C^\infty(p,Y):=C^\infty(U_p,Y)$ and $C^\infty(f_p^q)(g):=g\circ f$ for any $g\in C^\infty(U_q,Y)$.  The global sections of this sheaf $\CIN(X,Y)$ are precisely the smooth maps from $X$ to $Y$.  In particular, when $Y$ is an abelian diffeological group $G$, the result is a sheaf $C^\infty(\,\cdot\,,G)\colon\op{\calD(X)}\to \ab$ with global sections the group of smooth maps $C^\infty(X,G)$.

\item For a smooth map $\varphi\colon X\to Y$,  denote by $\Gamma(\cdot, \varphi)\colon\op{\calD(Y)} \to \set$ the sheaf given by $$\Gamma(p,\varphi):=\{\widetilde{p}\in C^\infty(U_p,X)\mid \varphi\circ \widetilde{p} =p\}$$ (\emph{i.e.}\ the set of lifts of plots of $Y$ to plots of $X$) and $\Gamma(f_q^p,\varphi)(\widetilde{p}):=\widetilde{p}\circ f$.  This sheaf is \emph{not} domain-determined.  The global sections of $\Gamma(X,\varphi)$ are precisely the smooth maps $\psi\colon Y\to X$ with $\varphi\circ \psi = \id_Y$.  In particular, global sections need not exist!
\end{enumerate}
\end{examples}

\begin{remark}\labell{r:sheaves}
Denote by $\euc$ the site of open subsets of Euclidean spaces with smooth maps between them and standard open covers.  There is a forgetful functor $\pi\colon\sl\calD(X)\to\euc$ sending a plot to its domain and $f_p^q$ to $f$.
In the language of sheaves over sites, a domain-determined $\Cc$-valued sheaf $F\colon\op{\calD(X)}\to \Cc$ is the pushforward of a sheaf $F'\colon\op{\euc}\to \Cc$ via the forgetful functor $\pi\colon\calD(X)\to \euc$.  Furthermore, as $\calD(X)$ itself is an instance of (the Grothendieck construction associated to) a sheaf of sets on $\euc$, global sections of $F$ can be identified with  maps of sheaves $\alpha\colon\calD(X)\Rightarrow F'$.  Of course, this is no longer true in the case of a sheaf that is not domain-determined.
\end{remark}

Our goal is to prove some basic properties of sheaves that one may expect from the classical case.  First, a definition.

\begin{definition}[Pullback Presheaf]\labell{d:pullback preshf}
Let $\varphi\colon X\to Y$ be a smooth map and $F\colon\op{\Cd(Y)}\to \Cc$ a $\Cc$-valued presheaf.  The {\em pullback presheaf $\varphi^*F\colon\op{\Cd(X)}\to \Cc$} is the presheaf with $\varphi^*F(p):=F(\varphi\circ p)$ on objects and $\varphi^*F(f_p^q):=F(f_{\varphi\circ p}^{\varphi\circ q})$ on morphisms.
\end{definition}

\begin{proposition}\labell{p:pullbackpracticalities}
Let $\varphi\colon X\to Y$ and $\psi\colon Y\to Z$ be smooth maps and $F\colon \op{\Cd(Y)}\to \Cc$, $G\colon\op{\calD(Z)}\to\calC$ presheaves.
\begin{enumerate}
	\item $\varphi$ induces a morphism $\widetilde{\varphi}\colon F(Y)\to \varphi^*F(X)$
	\item If $F$ is a sheaf over $Y$, then $\varphi^*F$ is a sheaf over $X$.
	\item $\varphi^*(\psi^*(G))=(\psi\circ \varphi)^*G$, and $\widetilde{\varphi}\circ \widetilde{\psi}=\widetilde{\varphi\circ \psi}$.
	\item If $F$ is domain-determined, then $\varphi^*F(p)=F(p)$.
	\item\label{i:functionsheavespb} For a diffeological space $W$, $\varphi^*C^\infty(\cdot ,W)=C^\infty(\cdot,W)$ and the induced map $\widetilde{\varphi}\colon C^\infty(Y,W)\to C^\infty(X,W)$ is given by precomposition by $\varphi$.
\end{enumerate}
\end{proposition}

\begin{proof}
Given a global section $\{\eta_p\}_{p\in \Cd(Y)}$, define $\widetilde{\varphi}(\{\eta_p\}_{p\in \Cd(Y)}):=\{\eta_{\varphi\circ q}\}_{q\in \Cd(X)}$.  Naturality follows from the definition.  This proves the first claim.

Let $p\in\calD_X$.  The restrictions $F|_{\varphi\circ p}$ and $\varphi^*F|_p$ over $\open(U_p)$ are identical presheaves.  As the former is a sheaf over $\open(U_p)$, so must be the latter.  This proves the second claim.

The remaining claims are straightforward to check.
\end{proof}

\begin{remark}\labell{r:domain-determined}
With respect to the setup and notation introduced in Remark~\ref{r:sheaves}, given a smooth map $\varphi\colon X\to Y$ and a domain-determined sheaf $F\colon\op{\calD(Y)}\to\calC$, we have that $\varphi^*F(p)=F'(U)$ for any $p\in\calD_X$.  Furthermore, the induced map between global sections is realized in this setting as the correspondence $$(\alpha\colon\calD(Y)\Rightarrow F')\mapsto(\alpha\circ\tilde{\varphi}\colon\calD(X)\Rightarrow F')$$ where $\tilde{\varphi}$ is the map of sheaves induced by $\varphi$.
\end{remark}

Next we define morphisms of sheaves, and what a short exact sequence of these is.


\begin{definition}[Morphism of Sheaves]\labell{d:morphsheaves}
Let $X$ be a diffeological space and $F,G\colon\op{\calD(X)}\to\calC$ sheaves.  A \emph{morphism of sheaves} $\Phi\colon F\Rightarrow G$ is a natural transformation between $F$ and $G$.  If $\calC$ is $\grp$, $\ab$, or $\vect$, then such a morphism of sheaves induces 
\begin{itemize}
	\item a morphism between global sections $\widetilde{\Phi}\colon F(X)\to G(X)$ sending a coherent family $\{\eta_p\}_{p\in\calD_X}$ to a coherent family $\{\Phi_p(\eta_p)\}_{p\in\calD_X}$;
	\item a \emph{kernel presheaf} $\ker\Phi\colon\op{\calD(X)}\to\calC$ defined by $\ker\Phi(p):=\ker(\Phi(p))$ for each plot $p\in\calD_X$, and $\ker\Phi(f_p^q):=F(f)|_{\ker\Phi(q)}$;
	\item an \emph{image presheaf} $\im\Phi\colon\op{\calD(X)}\to\calC$ defined by $\im\Phi(p):=\im(\Phi(p))$ for each plot $p\in\calD_X$, and $\im\Phi(f_p^q):=G(f)|_{\im(\Phi(q))}$.
	\item For any smooth $\varphi\colon Y\to X$, a morphism of sheaves $\varphi^*\Phi\colon\varphi^*F\Rightarrow\varphi^*G$.
\end{itemize}
\end{definition}

\begin{remark}\labell{r:morphsheaves}
Given a morphism of sheaves $\Phi\colon F\Rightarrow G$ as in Definition~\ref{d:morphsheaves}, the presheaves $\ker\Phi$ and $\im\Phi$ both send the empty plot $\ep$ to the terminal object of $\calC$.  Moreover, the kernel presheaf of a morphism of sheaves in the classical sense is in fact a sheaf, and so  $\ker\Phi$ is a sheaf in the diffeological sense.  However, in general, the image presheaf is not a sheaf, just as it is not in the classical theory.  This said, if $p$ is a plot, $U_p=\amalg_\alpha U_\alpha$, and $\mu=\Phi_p(\eta)$, then it follows from the naturality of $\Phi$ that $\mu=\sum_\alpha\Phi_{p_\alpha}(\eta_\alpha)$ where $i_\alpha\colon U_\alpha\to U$ is the inclusion, $p_\alpha=i_\alpha^*p$, and $\eta_\alpha=i_\alpha^*\eta$.  That is, image presheaves are completely determined by their values on plots with connected domains (\emph{cf}.\ Remark~\ref{r:coprods}).
\end{remark}

We follow Raeburn and Williams \cite{RaeburnWilliams} for our definition of a short exact sequence of sheaves, adapting to our diffeological context.

\begin{definition}[Short Exact Sequence]\labell{d:ses}
Let $X$ be a diffeological space and $F$, $G$, and $H$ be sheaves over $X$ with morphisms $\Phi\colon F\Rightarrow G$ and $\Psi\colon G\Rightarrow H$.  The sequence $$0 \Rightarrow F \overset{\Phi}{\Rightarrow} G \overset{\Psi}{\Rightarrow} H \Rightarrow 0$$
is a \emph{short exact sequence} if 
	\begin{enumerate}
		\item $\ker\Phi_p=0$ for each $p\in\calD_X$,
		\item $\ker\Psi_p=\im\Phi_p$ for each $p\in\calD_X$, and
		\item for every $p\in\calD_X$, $\eta\in H(p)$, and $u\in U_p$, there exists an open neighbourhood $V\subseteq U_p$ of $u$ such that $i_V^*\eta\in\im\Psi_{p|_V}$ where $i_V$ is the inclusion.
	\end{enumerate}
\end{definition}

The following lemma is straightforward to check, following from the definitions.

\begin{lemma}\labell{l:ses}
Let $\varphi\colon X\to Y$ be a smooth map, and $$0 \Rightarrow F \overset{\Phi}{\Rightarrow} G \overset{\Psi}{\Rightarrow} H \Rightarrow 0$$ a short exact sequence of sheaves of $Y$.  Then $$0 \Rightarrow \varphi^*F \overset{\varphi^*\Phi}{\Rightarrow} \varphi^*G \overset{\varphi^*\Psi}{\Rightarrow} \varphi^*H \Rightarrow 0$$ is a short exact sequence of sheaves over $X$. 
\end{lemma}

We also immediately obtain a familiar result from classical sheaf theory; namely, taking global sections is right exact.

\begin{proposition}\labell{p:ses}
Let $X$ be a diffeological space.  A short exact sequence of sheaves over $X$ $$0\Rightarrow F\overset{\Phi}{\Rightarrow} G\overset{\Psi}{\Rightarrow} H\Rightarrow 0$$ induces an exact sequence $0\to F(X)\overset{\widetilde{\Phi}}{\longrightarrow} G(X)\overset{\widetilde{\Psi}}{\longrightarrow} H(X)$.
\end{proposition}

\begin{proof}
Since $\ker\Psi_p=\im\Phi_p$ for each $p\in\calD_X$, it follows that $\widetilde{\Psi}\circ\widetilde{\Phi}=0$.  Injectivity of $\widetilde{\Phi}$ follows from the fact that $\ker\Phi_p=0$ for each $p\in\calD_X$.  Suppose $\{\mu_p\}\in\ker\widetilde{\Psi}$.  For each $p\in\calD_X$, $\Psi_p(\mu_p)=0$.  Since $\ker\Psi_p=\im\Phi_p$ for each $p$, there exists $\eta_p\in F(p)$ such that $\mu_p=\Phi_p(\eta_p)$.  Coherence of $\{\eta_p\}$ follows from $\Phi$ being a natural transformation, coherence of $\{\mu_p\}$, and injectivity of $\Phi_p$ for each $p$.  Thus $\{\mu_p\}=\widetilde{\Phi}(\{\eta_p\})$.
\end{proof}

\begin{example}\labell{x:morphsheaves}
Let $X$ be a diffeological space and $\alpha\colon A\to B$ a smooth homomorphism of abelian diffeological groups.  Then $\alpha$ induces a morphism of sheaves $\alpha_\sharp\colon\CIN(\cdot,A)\Rightarrow\CIN(\cdot,B)$.  This, in turn, induces a morphism of global sections $\widetilde{\alpha}\colon\CIN(X,A)\to\CIN(X,B)$.  Suppose $\alpha$ is an induction, and $\beta\colon B\to C$ is a subductive smooth homomorphism of abelian diffeological groups such that $$0\to A\underset{\alpha}{\longrightarrow} B\underset{\beta}{\longrightarrow} C\to 0$$ is a short exact sequence of groups.  It is straightforward to check that $\ker(\alpha_\sharp)_p=0$ for each $p\in\calD_X$.  Since $\alpha$ is an induction, it follows that $\ker(\beta_\sharp)_p=\im(\alpha_\sharp)_p$ for each $p\in\calD_X$.  Since $\beta$ is a subduction, for any $p\in\calD_X$, for any $h\in\CIN(U_p,C)$, and any $u\in U_p$, there exist an open neighbourhood $V$ of $u$ and $g\in\CIN(V,B)$ such that $h|_V=\beta\circ g$; that is, $i_V^*h\in\im(\beta_\sharp)_{p|_V}$ where $i_V$ is the inclusion.  It now follows from Proposition~\ref{p:ses} that $$0\to\CIN(X,A)\to\CIN(X,B)\to\CIN(X,C)$$ is exact.
\end{example}

\section{\v{C}ech cohomology for sheaves over diffeologies}\labell{s:cech}

In this section, we describe a model for \v{C}ech cohomology for (pre)sheaves of abelian groups over a fixed diffeological space.  Similar to classical \v{C}ech cohomology, we will first take the cohomology of a cochain complex associated to a specific generating family of the diffeological space in question and then take a colimit over the set of all generating families up to refinement.  

\begin{definition}\labell{d:k-fold nebula}
Let $X$ be a diffeological space and $\calQ$ a generating family.  For all $k\geq 1$, denote by $\Cn(\calQ)_k$ the $k$-fold fiber product
\[\Cn(\calQ)_k:=\{(x_0,\ldots, x_k)\in \Cn(\calQ)^{k+1}\,|\,\ev(x_0)=\ev(x_1)=\ldots=\ev(x_k)\}\]
endowed with the $k$-fold fiber product diffeology (see Definition~\ref{d:fiber product}).

For each $k\geq 1$, $\Cn(\calQ)_k$ comes with $k+1$ (smooth) {\em degeneracy maps} $d_i\colon\Cn(\calQ)_k\to \Cn(\calQ)_{k-1}$ defined by $d_i(x_0,\ldots, x_k):=(x_0,\ldots, \widehat{x_i},\ldots, x_k)$ (\emph{i.e.}\ the $i^{\mathrm{th}}$ factor is dropped).

For each $k\geq 0$, let $\ev_k\colon\Cn(\calQ)_k\to X$ be the (smooth) map $(x_0,\ldots, x_k)\mapsto \ev(x_0)$.
\end{definition}

\begin{remark}\labell{r:factorizations}
Plots of $\Cn(\calQ)_{k}$ are $(k+1)$-tuples $(f_0,\ldots,f_k)$ of plots of $\Cn(\calQ)$ satisfying $p:=\ev\circ f_0=\dots=\ev\circ f_k$. It follows that the tuple $(f_0,\ldots,f_k)$ is a choice of $k+1$ factorizations of $p$ through $\Cn(\calQ)$. In fact, any $k+1$ factorizations of a plot of $X$ through elements of $\calQ$ give rise to a plot of $\calN(\calQ)_k$ in this way.  Moreover, $\Cn(\calQ)$ is equipped with the sum diffeology giving it the structure of a local manifold, and so any plot $f\colon U\to \Cn(\calQ)$ in which $U$ is connected has image contained in a single plot domain $U_q$ of a plot $q$ in $\calQ$.  Hence, in the case that $U$ is connected, the $k+1$ factorizations of $p$ above are factorizations through elements of $\calQ$: if each $f_i$ has image in $U_{q_i}$, then we have the following commutative diagram.
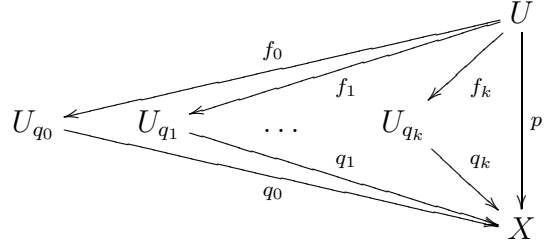
\[\xymatrix{& & & &  U \ar[dd]^{p} \ar[dllll]_{f_0} \ar[dlll]^{f_{1}} \ar[dl]^{f_k} \\ U_{q_0}\ar[drrrr]_{q_0}  & U_{q_1} \ar[drrr]^{q_1} & \ldots & U_{q_k} \ar[dr]^{q_k} &  \\ & & & &   X}\]
\end{remark}

Define the \v{C}ech cochain complex associated to a generating family as follows:

\begin{definition}[\v{C}ech Cohomology]\labell{d:cech}
Let $X$ be a diffeological space, $\calQ$ a generating family for $X$, and $F$ a presheaf of abelian groups on $X$.  A {\em degree $k$-cochain of $F$ associated to $\calQ$} is a global section of the sheaf $\ev_k^*F\colon\op{\calD(\Cn(\calQ)_k)}\to \ab$.  Denote by $\check{C}^k(\calQ,F)$ the group of all such global sections.  The degeneracy maps $d_i$ induce the following maps on global sections (see Proposition~\ref{p:pullbackpracticalities}):
$$\widetilde{d}_i\colon\ev_k^*F(\calN(\calQ)_k)\to d_i^*(\ev_k^*F)(\calN(\calQ)_{k+1})=\ev_{k+1}^*F(\calN(\calQ)_{k+1}).$$
The map $\del\colon\check{C}^k(\calQ,F)\to \check{C}^{k+1}(\calQ,F)$ defined as $\del:=\sum_{i=0}^{k+1} (-1)^{i}\widetilde{d}_i$ satisfies $\del^2=0$, and thus is a coboundary operator making $(\check{C}^*(\calQ,F),\del)$ a cochain complex, the \emph{\v{C}ech cochain complex associated to $\calQ$}.
The resulting cohomology $\check{H}^\bullet(\calQ,F)$ is the \emph{\v{C}ech cohomology for $F$ associated to $\calQ$}.
\end{definition}

\begin{remark}\labell{r:cech}
Suppose $F$ is a sheaf.  In the context of Remark~\ref{r:factorizations} and setting of Definition~\ref{d:cech}, if $(f_0,\ldots,f_k)$ is a $(k+1)$-tuple of factorizations of a plot $p\in\calD_X$ with connected domain through elements of $\calQ$ (and hence is a plot of $\calN(\calQ)_{k+1}$), then $\ev_k^*F(f_0,\ldots,f_k)=F(p)$.  As $F$ is determined by its values over plots of $X$ with connected domains (see Remarks~\ref{r:coprods} and \ref{r:factorizations}), $\check{C}^k(\calQ,F)$ is determined by coherent families of objects over the $(k+1)$-tuples of factorizations of each of these plots through elements of $\calQ$.  In terms of tuples of factorizations, $\del$ is realized as the map
\begin{equation}\label{e:cech}
\left\{\eta_{(f_0,\ldots,f_k)}\right\}_{\Cd(\Cn(\calQ)_k)_0}\mapsto \left\{\sum_{i=0}^{k+1}(-1)^i \eta_{(f_0,\ldots, \hat{f}_i,\ldots f_{k+1})}\right\}_{\calD(\calN(\calQ)_{k+1})_0}.
\end{equation}
\end{remark}

For sheaves of smooth functions to an abelian diffeological group, \v{C}ech cohomology reduces to a simple form.

\begin{example}\labell{x:mappingsheaves}
Let $G$ be an abelian diffeological group and $C^\infty(\cdot ,G)\colon\op{\Cd(X)}\to \ab$ the domain-determined sheaf of smooth maps defined in the second example of Examples~\ref{x:sheaves}.  By Item~\ref{i:functionsheavespb} of Proposition~\ref{p:pullbackpracticalities}, for any choice of generating family $\calQ$ of $X$, the sheaf of smooth functions on $X$ pulls back to the sheaf of smooth functions on $\calN(\calQ)_k$; that is, $\ev_k^*C^\infty(\cdot,G)=C^\infty(\cdot,G)$.  Furthermore, the maps $\widetilde{d}_i\colon \ev_k^*C^\infty(\calN(\calQ)_k ,G)\to \ev_{k+1}^*C^\infty(\calN(\calQ)_{k+1} ,G)$ are given by precomposition with $d_i$.  

The $k$-cochains $\check{C}^k(\calQ,G):=\check{C}^k(\calQ,C^\infty(\cdot, G))$ thus reduce to smooth maps $C^\infty(\Cn(\calQ)_k,G)$ and $\partial$ reduces to 
\[\partial f:=\sum_{i=0}^{k+1}(-1)^k f\circ d_i\]
for $f\in\check{C}^k(\calQ,G)$.
\end{example}

The degree zero \v{C}ech cohomology associated to {\em any} generating family yields the global sections of a sheaf $F$.

\begin{proposition}\labell{p:deg-0}
Let $X$ be a diffeological space, $\calQ$ a generating family, and $F$ a sheaf of abelian groups on $X$.  Then $\check{H}^0(\calQ,F)=F(X)$.
\end{proposition}

\begin{proof}
By Remarks~\ref{r:factorizations} and \ref{r:cech}, it suffices to assume that domains of plots in $\calQ$ are connected, in which case a degree-$0$ cochain of $F$ associated to $\calQ$ is given by a global section $\{\eta_f\}_{f\in \calD(\calN(\calQ))_0}$.  This cochain is a cocycle if, for every plot $p\colon U\to X$, open cover $\{U_\alpha\}$ of $U$, and factorization $p|_{U_\alpha}=q_\alpha\circ f_\alpha$ with $q_\alpha\colon V_\alpha\to X$ in $\calQ$ and $f_\alpha\colon U_\alpha\to V_\alpha$ smooth, $\eta_{f_{\alpha_1}}$ and $\eta_{f_{\alpha_2}}$ coincide on $U_{\alpha_1}\cap U_{\alpha_2}$ for every $\alpha_1,\alpha_2$.

Any global section $\{\eta_p\}_{p\in \calD_X}$ induces a global section $\{\eta_f\}_{f\in\calD(\calN(\calQ))_0}$ via restriction.  Conversely, given a cocycle $\{\xi_f\}_{f\in\calD(\calN(\calQ))_0}$, we extend it to a global section of $F$ as follows: for a plot $p\colon U\to X$, let $\{U_\alpha\}$ be an open cover of $U$, $\{q_\alpha\colon V_\alpha\to X\}\subset\calQ$, and $\{f_\alpha\colon U_\alpha\to V_\alpha\}$ a family of smooth functions such that $p|_{U_\alpha}=q_\alpha\circ f_\alpha$ for all $\alpha$.  Define $\xi_{p|_{U_\alpha}}:=\xi_{f_\alpha}$.  For any $\alpha_1,\alpha_2$, 
$$q_{\alpha_1}\circ f_{\alpha_1}|_{U_{\alpha_1}\cap U_{\alpha_2}}=p|_{U_{\alpha_1}\cap U_{\alpha_2}}=q_{\alpha_2}\circ f_{\alpha_2}|_{U_{\alpha_1}\cap U_{\alpha_2}},$$ from which the cocycle condition implies that 
\[(\xi_{p|_{U_{\alpha_1}}})|_{U_{\alpha_1}\cap U_{\alpha_2}} =\xi_{f_{\alpha_1}}|_{U_{\alpha_1}\cap U_{\alpha_2}} = \xi_{f_{\alpha_2}}|_{U_{\alpha_1}\cap U_{\alpha_2}} = (\xi_{p|_{U_{\alpha_2}}})|_{U_{\alpha_1}\cap U_{\alpha_2}}. \]
Since $F$ restricts to a sheaf on $\open(U)$ in the classical sense, the family $\{\xi_{p|_{U_\alpha}}\}$ glues together into a unique object $\xi_p\in F(p)$.  It follows from the cocycle condition that $\xi_p$ is independent of the choices of $\{q_\alpha\}$ and $\{f_\alpha\}$.  A similar argument yields the coherence of $\{\xi_p\}_{p\in\calD_X}$.
\end{proof}

To define general \v{C}ech cohomology of a diffeological space, one takes a limit of \v{C}ech cohomology over refinements of generating families, which we define now.

\begin{definition}[Refinement of a Generating Family]\labell{d:refinement}
Let $X$ be a diffeological space with generating family $\calQ=\{q_\alpha\colon U_\alpha\to X\}_{\alpha\in A}$.  A {\em refinement of $\calQ$} is a generating family $\calR=\{r_\beta\colon V_\beta\to X\}_{\beta\in B}$ with a map of index sets $\varphi\colon B\to A$ such that for every plot $r_\beta\colon V_\beta\to X$ of $\calR$, there exists a smooth map $f_\beta\colon V_\beta\to U_{\varphi(\beta)}$ satisfying $q_{\varphi(\beta)}\circ f_\beta=r_\beta.$

Given two generating families $\calQ$ and $\calR$ of $X$, a {\em common refinement of $\calQ$ and $\calR$} is a third generating family $\calS$ which refines both $\calQ$ and $\calR$.
\end{definition}

\begin{lemma}\labell{l:alwayscrs}
Let $X$ be a diffeological space with generating families $\calQ$ and $\calR$.  There exists a common refinement of $\calQ$ and $\calR$.
\end{lemma}

\begin{proof}
Let $\calQ=\{q_\alpha\colon U_\alpha\to X\}_{\alpha\in A}$ and $\calR=\{r_\beta\colon V_\beta\to X\}_{\beta\in B}$.  Since $\calR$ is a generating family, for each $q_\alpha\in\calQ$ there is an open cover $\{U^\alpha_\gamma\}_{\gamma\in\Gamma_\alpha}$ of $U_\alpha$, plots $\{r_{\beta(\alpha,\gamma)}\colon V_{\beta(\alpha,\gamma)}\to X\}\subset\calR$, and a family of smooth maps $\{f^\alpha_\gamma\colon U^\alpha_\gamma\to V_{\beta(\alpha,\gamma)}\}$ such that $q_\alpha|_{U^\alpha_\gamma}=r_{\beta(\alpha,\gamma)}\circ f^\alpha_\gamma$ for every $\gamma$.  Define $\calS$ to be the collection $\{q_\alpha|_{U^\alpha_\gamma}\}_{(\alpha,\gamma)}$ as $\alpha$ runs over $A$ and $\gamma$ over $\Gamma_\alpha$.  Then $\calS$ is a generating family, a refinement of $\calQ$ via the map $(\alpha,\gamma)\mapsto\alpha$ and the inclusions $U^\alpha_\gamma\to U_\alpha$, and a refinement of $\calR$ via the map $(\alpha,\gamma)\mapsto\beta(\alpha,\gamma)$ and the maps $f^\alpha_\gamma$.
\end{proof}

Using the notation of Definition~\ref{d:refinement}, given a refinement $\calR$ of a generating family $\calQ$ of $X$, there is an induced smooth map $f\colon\calN(\calR)\to\calN(\calQ)$ given by $f:=\coprod_\beta f_\beta$.  It is straightforward to show:

\begin{lemma}\labell{l:refinement}
Fix a diffeological space $X$.
\begin{enumerate}
\item\labell{i:ev comm}Let $f\colon\calN(\calR)\to\calN(\calQ)$ be the map induced by a refinement $\calR$ of a generating family $\calQ$ of $X$.  If $\ev_\calQ$ and $\ev_\calR$ are the evaluation maps of $\calN(\calQ)$ and $\calN(\calR)$, respectively, then $\ev_\calQ\circ f = \ev_\calR$.  
\item If $\calQ$ and $\calR$ are two generating families of $X$ with common refinement $\calS$ and induced maps $f\colon\calN(\calS)\to\calN(\calQ)$ and $f'\colon\calN(\calS)\to\calN(\calR)$, then $\ev_\calQ\circ f = \ev_\calR\circ f'$.
\end{enumerate}
\end{lemma}

\begin{remark}\labell{r:refinement}
A partial converse is true: given a diffeological space $X$, if $\calQ$ is a generating family, $\calR$ a generating family whose plots have connected domains, and $f\colon\Cn(\calR)\to \Cn(\calQ)$ a smooth map satisfying $\ev_\calQ\circ f = \ev_\calR$, then $\calR$ is a refinement of $\calQ$ with induced map $f$.
\end{remark}

Refinements of generating families induce chain maps between cohomology complexes.

\begin{lemma}\labell{l:refinementtocoho}
Let $X$ be a diffeological space, $\calQ$ a generating family with refinement $\calR$, and $F\colon\op{\Cd(X)}\to \ab$ a presheaf of abelian groups on $X$.  The induced map $f\colon\calN(\calR)\to\calN(\calQ)$ induces a map $\check{f}\colon \check{H}^k(\calQ,F) \to \check{H}^k(\calR,F).$
\end{lemma}

\begin{proof}
Define $f^k\colon\Cn(\calR)_k\to\Cn(\calQ)_k$ for each $k$ by $f^k(x_0,\ldots, x_k):=(f(x_0),\ldots,f(x_k))$.  For each $k$, the degeneracy maps $d_i$ satisfy the equality $d_i\circ f^k=f^{k-1}\circ d_i$.  It follows that the resulting chain maps $f^\sharp\colon(\check{C}^*(\calQ,F),\del)\to (\check{C}^*(\calR,F),\del)$ induce a map of cohomology groups $\check{f}\colon\check{H}^k(\calQ,F)\to \check{H}^k(\calR,F)$.
\end{proof}

We are now ready to give the definition for \v{C}ech cohomology of a diffeological space $X$.

\begin{definition}\labell{d:cech2}
Let $X$ be a diffeological space and $F\colon\op{\Cd(X)}\to \ab$ a presheaf of abelian groups on $X$.  The {\em degree-$k$ \v{C}ech cohomology group of $F$ over $X$}, denoted $\check{H}^k(X,F)$, is the directed limit of the set of cohomology classes $\check{H}^k(\calQ,F)$ partially ordered by refinement; that is, the colimit over the diagram of morphisms $\check{f}\colon \check{H}^k(\calQ,F)\to \check{H}^k(\calR,F)$ induced by refinement (see Lemma~\ref{l:refinementtocoho}).
\end{definition}

\v{C}ech cohomology is natural.

\begin{proposition}\labell{p:functorial}
Let $\varphi\colon X\to Y$ be smooth and $F\colon\op{\Cd(Y)}\to \ab$ a presheaf of abelian groups over $Y$.  Then $\varphi$ induces a map on cohomology $\varphi^*\colon\check{H}^k(Y,F)\to \check{H}^k(X,\varphi^*F)$.
\end{proposition}

\begin{proof}
Fix $a\in \check{H}^k(Y,F)$.  There is a generating family $\calQ$ for $Y$ such that $a\in \check{H}^k(\calQ,F)$.  Let $\calR$ be a generating family for $X$.  For each $q$ in $\calR$, since $\varphi\circ q$ is a plot of $Y$, there is an open cover $\{U_\alpha\}$ of $U_q$ such that each restriction $\varphi\circ q|_{U_\alpha}$ factors through an element of $\calQ$.  The inclusions $U_\alpha\to U_q$ induce a refinement $\calR'$ of $\calR$, which contains the restrictions $q|_{U_\alpha}$ for each $\alpha$ and each $q\in\calR$.

By construction, there is a smooth map $\psi\colon\Cn(\calR')\to \Cn(\calQ)$ such that $\ev\circ\psi=\varphi\circ\ev$. For each $k$, $\psi$ induces a map $\psi^k\colon\calN(\calR')_k\to\calN(\calQ)_k$ sending $(x_0,\dots,x_k)$ to $(\psi(x_0),\dots,\psi(x_k))$ satisfying $\ev_k\circ \psi^k=\varphi\circ\ev_k$.  Moreover, with respect to the degeneracy maps the equality $d_i\circ\psi^k=\psi^{k-1}\circ d_i$ holds for all $k$ and $i$.  The resulting chain maps $\psi^\sharp\colon\check{C}^\bullet(\calQ,F)\to \check{C}^\bullet(\calR',\varphi^*F)$ descend to a map of cohomology groups $\check{\psi}\colon\check{H}^k(\calQ,F)\to \check{H}^k(\calR',\varphi^*F)$. 

Suppose $\calQ'$ is a refinement of $\calQ$. As above, there exists a refinement $\calR''$ of $\calR'$ in which each plot of $\calR''$ factors through an element of $\calQ'$. With respect to the induced maps between refinements and nebulae, the resulting commutative diagram 
\[
\xymatrix{
\Cn(\calR'') \ar[r] \ar[d] &\Cn(\calQ')  \ar[d] \\
\Cn(\calR')\ar[r] & \Cn(\calQ) ,
}
\]
induces the commutative diagram by Lemma~\ref{l:refinementtocoho}
\[
\xymatrix@R=1em{
& \check{H}(\calR';\varphi^*F) \ar[dd] \ar[ld] & \check{H}(\calQ;F) \ar[l] \ar[dd] \\
\check{H}(X;\varphi^*F) && \\
& \check{H}(\calR'';\varphi^*F) \ar[lu]& \check{H}(\calQ';F). \ar[l] 
}
\]
To conclude, we obtain a collection of maps $\check{H}(\calQ;F) \to \check{H}(X;\varphi^*F)$ for every generating family $\calQ$ of $Y$, such that if $\calQ'$ is a refinement of $\calQ$, the diagram
\[
\xymatrix@R=1em{
&  \check{H}(\calQ;F) \ar[ld] \ar[dd] \\
\check{H}(X;\varphi^*F) && \\
&  \check{H}(\calQ';F) \ar[lu] 
}
\]
commutes.  The universal property of colimits gives the required induced map on direct limits.
 \end{proof}
 
 We end this section with showing that morphisms between sheaves induce maps between \v{C}ech cohomology groups.
 
\begin{proposition}\labell{p:morph cohom}
Let $X$ be a diffeological space and $F,G$ sheaves of abelian groups on $X$. A map of sheaves $\Phi\colon F\Rightarrow G$ induces a homomorphism  $\check{\Phi}\colon\check{H}^k(X;F)\to\check{H}^k(X;G)$.
\end{proposition}

\begin{proof}
Fix a map of sheaves $\Phi\colon F\Rightarrow G$, and let $\widetilde{\Phi}\colon F(X)\to G(X)$ be the induced map on global sections (see Definition~\ref{d:morphsheaves}).  Fix a generating family $\calQ$ of $X$.  By Remarks~\ref{r:factorizations} and \ref{r:cech}, $\widetilde{\Phi}$, in turn, induces a map $\Phi^k_\sharp\colon\check{C}^k(\calQ,F)\to\check{C}^k(\calQ,G)$ where $$\Phi^k_\sharp\left(\{\eta_{(f_0,\dots,f_k)}\}\right)=\{\Phi_{\ev_k\circ(f_0,\dots,f_k)}(\eta_{(f_0,\dots,f_k)})\}.$$
By \eqref{e:cech}, $\Phi^{k+1}_\sharp\circ\del=\del\circ\Phi^{k}_\sharp$, and so we obtain a map between \v{C}ech cohomology groups associated to $\calQ$.  The remainder of the proof is similar to that of Proposition~\ref{p:functorial}, from which we obtain a map $\check{\Phi}\colon\check{H}^k(X;F)\to\check{H}^k(X;G)$.
\end{proof}


\section{Classifying Principal $G$-bundles}\labell{s:bundles}

In this section, we fix a diffeological group $G$ and prove that in the abelian case, the isomorphism classes of principal $G$-bundles over a diffeological space $X$ are in natural bijection with the cohomology classes $\check{H}^1(X,G):=\check{H}^1(X,C^\infty(\cdot, G))$.  We begin by defining a principal $G$-bundle:

\begin{definition}[Principal $G$-Bundle]\labell{d:principal bundle}
Given a diffeological space $X$, a (right) {\em principal $G$-bundle over $X$} is a diffeological space $P$ with a smooth right action of $G$ and quotient map $\pi\colon P\to X$, for which the action map
\[a\colon P\times G\to P\times_X P,\; (p,g)\mapsto (p,p\cdot g)\]
is a diffeomorphism. We shall denote a principal bundle by $\pi: P \to X$.

A {\em map  $P \to P'$ of principal $G$-bundles over $X$} is a $G$-equivariant map intertwining the quotient maps to $X$.

A principal $G$-bundle over $X$ that is isomorphic to $X\times G$ will be called {\em trivializable}.
\end{definition}

\begin{remark}\labell{r:principal bundle}
This definition of principal bundle is equivalent to the one that appears in \cite[Article 8.11]{PIZ}.  Indeed, the latter definition is essentially the same except that it requires the action map $P\times G\to P\times P$ be an induction.  However, the action map is injective with image $P\times_X P$, which is equipped with the subset diffeology induced by $P\times P$.  The equivalence follows.
\end{remark}

The following proposition is a collection of facts that are either straightforward to check and/or can be found in \cite[Articles 8.11-12]{PIZ}.

\begin{proposition}\labell{p:bdle prop}
Let $\pi\colon P\to X$ be a principal $G$-bundle.
\begin{enumerate}
\item The \emph{division map} $d:=\pr_G\circ a^{-1}$ is smooth.
\item \labell{i:torsor fiber}The action of $G$ on $P$ is free and transitive on fibres of $\pi$.
\item \labell{i:triv} $\pi$ is trivializable if and only if there exists a global section $\sigma\colon X\to P$ of $\pi$.
\item Given a smooth map $\varphi\colon Y\to X$, the pullback $\varphi^*P:=P\times_X Y\to Y$ is a principal $G$-bundle.
\end{enumerate}
\end{proposition}

\jnote{removed:\\

\begin{proof}
The first two claims are immediate from the definition of a principal bundle and its action map.

For the third claim, suppose $P$ is trivializable.  The inclusion $X\to X\times G$ sending $x$ to $(x,1_G)$, composed with a bundle isomorphism from $X\times G$ to $P$, is a global section of $P$.

Conversely, suppose that $\sigma\colon X\to P$ is a smooth section of $\pi$.  The projection $\pr_2\circ a\circ (\sigma\times\id_G)\colon X\times G\to P$ is a smooth bijection, with inverse given by $(\pi\times\id_G)\circ a^{-1}\circ(\sigma\circ\pi,\id_P)$.  This proves the third claim.

The action of $G$ on $P$ and the trivial action of $G$ on $Y$ induce an action of $G$ on $P\times_X Y$.  Let $p\in\calD_Y$.  Since $\pi$ is a subduction, there exist an open cover $\{U_\alpha\}$ of $U_p$ and local lifts $q_\alpha\colon U_\alpha\to P$ of $\varphi\circ p$.  The plot $\pr_Y\circ(q_\alpha,p|_{U_\alpha})\colon U_\alpha\to P\times_X Y$ is a local lift of $p$ to $P\times_XY$.  As $p$ is arbitrary, this shows that $\pr_Y$ is a subduction.  The fibers of $\pr_Y$ are exactly the orbits of $G$, and so $\pr_Y$ is the quotient map of the $G$-action.

Define $b\colon (P\times_X Y)\times G\to (P\times_X Y)\times_Y (P\times_X Y)$ to be the smooth map $b(p,y,g):=(p,y,p\cdot g,y)$.  By Claim~\ref{i:torsor fiber}, the action of $G$ on $P$ is free and transitive on fibers of $\pi$, from which it follows that $b$ is bijective.  An inverse of $b$ is given by $(p_1,y,p_2,y)\mapsto(p_1,y,d(p_1,p_2))$ where $d$ is the division map of $P$.  As $b^{-1}$ is smooth, we conclude that $b$ is a diffeomorphism, and $P\times_X Y\to Y$ is a principal $G$-bundle.  This completes the proof.
\end{proof}

}

In general, principal $G$-bundles are not locally trivial with respect to the natural topology on their base (called the \emph{$D$-topology}); see \cite[Article 8.9]{PIZ}.  However, we get another form of local trivialization more appropriate for diffeology.

\begin{definition}\labell{d:plotwise local triv}
A {\em plotwise local trivialization} of a principal $G$-bundle $\pi\colon P\to X$ is a generating family $\calQ$ for $X$ and a smooth map $\tau\colon\Cn(\calQ)\to P$ satisfying $\pi\circ \tau=\ev$. 
\end{definition}

\begin{proposition}\labell{p:loc triv prop}
Let $\pi\colon P\to X$ be a principal $G$-bundle.
\begin{enumerate}
\item\labell{i:loc triv exist}$\pi$ admits a plotwise local trivialization.
\item\labell{i:pbtriv}Given a plotwise local trivialization $\tau\colon\calN(\calQ)\to P$, the pullback $\ev^*P\to\calN(\calQ)$ is trivializable.
\end{enumerate}
\end{proposition}

\begin{proof}
The first claim follows from Example~\ref{x:nebula}, where $\tau=\widetilde{\ev}$.

For the second claim, define a section $\sigma\colon\Cn(\calQ)\to \ev^*P$ by $\sigma(x):=(\tau(x),x)$.  The proof follows from Claim~\ref{i:triv} of Proposition~\ref{p:bdle prop}.
\end{proof}

Claim~\ref{i:loc triv exist} of Proposition~\ref{p:loc triv prop} reflects the necessary and sufficient condition for a principal $G$-bundle found in \cite[Article 8.13]{PIZ}: pullback bundles via plots are principal $G$-bundles in the classical sense; namely, they are locally trivial with respect to the topology of their bases (which are manifolds).

Continuing toward our goal of classifying principal $G$-bundles using degree-$1$ \v{C}ech cohomology, we next build a degree-$1$ \v{C}ech cocycle from any principal bundle given a choice of plotwise local trivialization.  Henceforth, assume that $G$ is an \emph{abelian} diffeological group. 

\begin{construction}\label{cons:cocyclefrombundle}
Given a principal $G$-bundle $\pi\colon P\to X$ and a plotwise local trivialization $\tau\colon\calN(\calQ)\to P$, define the $1$-cochain $c(\tau, P)\colon\Cn(\calQ)_1\to G$ in $\check{C}^1(\calQ,G)$ (see Example~\ref{x:mappingsheaves}) to be the composite
\begin{equation}\labell{eq:bundletococycle}
\xymatrix{\Cn(\calQ)_1\ar[r]^-{\tau\times\tau} \ar@/_1pc/[rr]_{c(\tau,P)} & P\times_X P \ar[r]^-{d} & G}
\end{equation} 
where $d$ is the division map.
\end{construction}

\begin{lemma}\labell{l:cocycle}
The $1$-cochain $c(\tau,P)$ above is a $1$-cocycle.
\end{lemma}

\begin{proof}
For all $x,y,z$ in the same fiber of $\ev$,
$$\partial c(\tau,P)(x,y,z) = c(\tau,P)(y,z)-c(\tau,P)(x,z)+c(\tau,P)(x,y).$$
By the definition of the division map $d$ in \eqref{eq:bundletococycle}, the right-hand side is equal to $0$.
\end{proof}

To prove the bijection of \v{C}ech cohomology and isomorphism classes of principal bundles, we will use the language of groupoid actions as an intermediary.  The groupoids are not necessarily Lie groupoids, but are in the diffeological category.

\begin{definition}[Diffeological Groupoid]\labell{d:diffeol gpd}
A \emph{diffeological groupoid} is a groupoid $H=(H_1\rra H_0)$ in which $H_1$ and $H_0$ are diffeological spaces, and all structure maps are smooth.
\end{definition}

It follows from the definition that the unit map $u\colon H_0\to H_1$ is an induction, and consequently, the source $s$ and target $t$ maps are subductions.

\begin{definition}[Groupoid Action]\labell{d:gpd action}
Given a diffeological groupoid $H=(H_1\rra H_0)$ and a diffeological space $X$, a {\em (left) smooth action of $H$ on $X$ with (smooth) anchor map $\rho\colon X\to H_0$} is a smooth map $\act\colon H_1\ftimes{s}{q}X\to X\colon (h,x)\mapsto h\cdot x$ satisfying:
\begin{itemize}
	\item $\rho(h\cdot x)=t(h)$ for all $(h,x)$ with $s(h)=\rho(x)$,
	\item $g\cdot (h\cdot x) = (gh)\cdot x$ for all $g,h\in H_1$ and $x\in X$ with $s(g)=t(h)$ and $s(h)=\rho(x)$,
	\item $u(\rho(x))\cdot x = x$ for all $x\in X$.
\end{itemize}

Given two $H$-spaces $X$ and $Y$ with anchor maps $\rho_X\colon X\to H_0$ and $\rho_Y\colon Y\to H_0$, respectively, an {\em $H$-equivariant map} $\varphi\colon X\to Y$ is a smooth map satisfying:
\begin{itemize}
	\item $\rho_Y\circ \varphi =\rho_X$, and
	\item $\varphi(h\cdot x)=h\cdot \varphi(x)$.
\end{itemize}
(Note that the second condition implicitly requires the first.)
\end{definition}

\begin{example}\labell{x:gpd action}
Let $\pi\colon Y\to X$ be smooth.  Define a diffeological groupoid $R(\pi)=(Y\times_X Y\rra Y)$ with
\begin{itemize}
	\item source and target maps $s(y_1,y_2)=y_2$ and $t(y_1,y_2)=y_1$, respectively;
	\item multiplication $(y_1,y_2)(y_2,y_3)=(y_1,y_3)$;
	\item inversion map $(y_1,y_2)\mapsto (y_2,y_1)$; and
	\item unit map $y\mapsto (y,y)$.
\end{itemize}

It follows from Definition~\ref{d:gpd action} that any action of $R(\pi)$ on $Y\times G$ with anchor map $\pr_1\colon Y\times G\to Y$ that commutes with the action of $G$ on $Y\times G$ by (right) translation is of the form
$(y_1,y_2)\cdot(y_2,g)=(y_1,\varphi(y_1,y_2)+g)$
for some smooth $\varphi\colon Y\times_X Y\to G$ satisfying $\varphi(y_1,y_2)+\varphi(y_2,y_3)=\varphi(y_1,y_3)$.  If $Y=\calN(\calQ)$ for a generating family $\calQ$ of $X$ and $\pi=\ev$, then
for a fixed $1$-cocycle $f\in\check{C}^1(\calQ,G)$, we can define an action of $R(\ev)$ on $\calN(\calQ)\times G$ with anchor map $\pr_1\colon\calN(\calQ)\times G\to\calN(\calQ)$ by
\begin{equation}\labell{eq:cocycleaction}
(y_1,y_2)\cdot (y_2,g):= (y_1,f(y_1,y_2)+g).
\end{equation}
The following lemma clarifies the relationship between cocycles of $\check{C}^1(\calQ,G)$ and the above actions of $R(\ev)$.
\end{example}

\begin{lemma}\labell{l:groupoidrep}
Let $X$ be a diffeological space and $\calQ$ a generating family.  There is a bijection from $1$-cocycles of $\check{C}^1(\calQ,G)$ to the set of actions of $R(\ev)$ on $\calN(\calQ)\times G$ with anchor map $\pr_{\calN(\calQ)}$ that commute with the action of $G$ by right multiplication as defined in Example~\ref{x:gpd action}.  Moreover, fixing cocycles $f_1,f_2\in\check{C}^1(\calQ,G)$, there is a bijection between $0$-cochains $\alpha\in\check{C}^0(\calQ,G)$ such that $\partial\alpha=f_2-f_1$ and ($R(\ev)\text{-}G$)-equivariant maps from $\calN(\calQ)\times G$ to itself, where equivariance is with respect to the right $G$-action and the $R(\ev)$-actions induced by $f_1$ and $f_2$.
\end{lemma}

\begin{proof}
Suppose $f_1$ and $f_2$ in $\check{C}^1(\calQ,G)$ induce the same action on $\calN(\calQ)\times G$.  Setting $g=0_G$ in \eqref{eq:cocycleaction}, it follows that $f_1=f_2$.  On the other hand, let $\act\colon\Cn(\calQ)_1\times_{\Cn(\calQ)}(\Cn(\calQ)\times G)\to \Cn(\calQ)\times G$ be an action of $R(\ev)$ on $\calN(\calQ)\times G$ with anchor map $\pr_{\calN(\calQ)}\colon\calN(\calQ)\times G\to\calN(\calQ)$ that commutes with the right action of $G$ on $\calN(\calQ)\times G$ by right multiplication.  Denote by $\pr_G\colon \Cn(\calQ)\times G\to G$ the projection map onto $G$ and $\iota\colon\Cn(\calQ)_1\to \Cn(\calQ)_1\times_{\Cn(\calQ)}(\Cn(\calQ)\times G)$ the inclusion $(y_1,y_2)\mapsto ((y_1,y_2),(y_2,0_G))$.  By Example~\ref{x:gpd action}, the action satisfies \eqref{eq:cocycleaction} with $f$ defined to be $\pr_G\circ \act \circ \iota$, establishing the desired bijection.

Any $(R(\ev)\text{-}G)$-equivariant map $\Cn(\calQ)\times G\to \Cn(\calQ)\times G$ associated to the induced actions of $1$-cocycles $f_1,f_2\colon\Cn(\calQ)_1\to G$ is necessarily of the form $(y,g)\mapsto (y,g-\alpha(y))$ for some $0$-cochain $\alpha\colon\Cn(\calQ)\to G$.  This correspondence is injective.  On the other hand, fix a $0$-cochain $\alpha\colon\calN(\calQ)\to G$, and let $\tilde{\alpha}\colon \Cn(\calQ)\times G\to \Cn(\calQ)\times G$ be the smooth map $(y,g)\mapsto(y,g-\alpha(y))$.  Then 
\begin{align*}
\tilde{\alpha}((y_1,y_2)\cdot (y_2,g)) &=  (y_1,f_1(y_1,y_2)+g-\alpha(y_1)) \quad \text{and}\\
(y_1,y_2)\cdot\tilde{\alpha}(y_2,g) &= (y_1,f_2(y_1,y_2)+g-\alpha(y_2)).
\end{align*}
Hence, $\tilde{\alpha}$ is equivariant if and only if $f_1(y_1,y_2)-\alpha(y_1) = f_2(y_1,y_2)-\alpha(y_2)$ for all $(y_1,y_2)$; that is, if and only if $f_2-f_1=\partial\alpha$.  This shows surjectivity, completing the proof.
\end{proof}

\begin{remark}\labell{r:groupoidrep}
The $(R(\ev)\text{-}G)$-equivariant maps of Lemma~\ref{l:groupoidrep} are in fact diffeomorphisms. One can rephrase this lemma in the language of diffeological groupoids: the category of diffeological left principal bibundles from $R(\ev)$ to the group $G$ (viewed as the groupoid $G\toto *$), with ($R(\ev)$-$G$)-equivariant diffeomorphisms between the bibundles as arrows, is categorically equivalent to the category whose objects are $1$-cocycles of $\check{C}^1(\calQ,G)$, and arrows between $f_1,f_2\in\check{C}^1(\calQ,G)$ are $0$-cochains $\alpha$ in $\check{C}^0(\calQ,G)$ such that $f_2-f_1=d\alpha$.  For more information on diffeological groupoids, see \cite{vdS}.
\end{remark}

\begin{lemma}\labell{l:pbtogroupoid}
Let $\pi\colon P\to X$ be a principal $G$-bundle with plotwise local trivialization $\tau\colon\Cn(\calQ)\to P$.  The map
\[\varrho\colon\Cn(\calQ)\times G\to P\colon(x,g)\mapsto \tau(x)\cdot g\]
is a quotient map for the groupoid action given by Equation~\eqref{eq:cocycleaction} with $f=c(\tau,P)$.
\end{lemma}

\begin{proof}
By Claim~\ref{i:pbtriv} of Proposition~\ref{p:loc triv prop} (and its proof), $\ev^*P:=\Cn(\calQ)\times_{X}P\to \Cn(\calQ)$ is trivializable, with global section $\sigma\colon\Cn(\calQ)\to \Cn(\calQ)\times_{X}P\colon y\mapsto (y,\tau(y))$ and $G$-equivariant diffeomorphism $\varphi\colon\Cn(\calQ)\times G\to \Cn(\calQ)\times_X P, \, (y,g)\mapsto(y,\tau(y)\cdot g)$.  Let $p\colon U\to P$ be a plot.  There exist an open cover $\{U_\alpha\}$ of $U$ for which $U_\alpha$ is connected for each $\alpha$, a family $\{q_\alpha\}$ of plots in $\calQ$, and a family of smooth functions $\{f_\alpha\}$ such that $\pi\circ p|_{U_\alpha}=q_\alpha\circ f_\alpha$ for each $\alpha$.  The plot $(f_\alpha,p|_{U_\alpha})$ is a lift of $p|_{U_\alpha}$ to $\calN(\calQ)\times_X P$.  It follows that $\varrho=\pr_2\circ\varphi$ is a subduction.

To show that $\varrho$ is the quotient map of an $R(\ev)$-action, it remains to show that the fibers of $\varrho$ are precisely the orbits of the action of $R(\ev)$ on $\Cn(\calQ)\times G$.  Suppose $\varrho(y,g)=\varrho(y',g')$.  Then $(y,y')\in R(\ev)$, and $f(y,y')=g-g'$, and so $(y,y')\cdot(y',g')=(y,g)$.  The fibers of $\varrho$ are contained in orbits of $R(\ev)$.  On the other hand,  $\varrho$ is $R(\ev)$-invariant:
\[\varrho((y',y)\cdot (y,g)) = \varrho(y',c(P,\tau)(y',y)+g) = \tau(y')\cdot (c(P,\tau)(y',y)+g) = \tau(y)\cdot g = \varrho(y,g).\]
\end{proof}

\begin{lemma}\labell{l:groupoidtopb}
Given a subduction $\varpi\colon Y \to X$, and an action of $R(\varpi)$ on $Y\times G$ with anchor map $\pr_1\colon Y\times G\to Y$ that commutes with the action of $G$ by multiplication on the right, the map $\pi\colon P:=(Y\times G)/R(\varpi)\to X$ sending $[y,g]$ to $\varpi(y)$ is a principal $G$-bundle.  Moreover, if $\Phi\colon Y\times G\to Y\times G$ is an ($R(\varpi)\text{-}G$)-equivariant map with respect to two such actions of $R(\varpi)$ on $Y\times G$, and right multiplication by $G$, then $\Phi$ descends to an isomorphism between the resulting principal $G$-bundles.
\end{lemma}   

\begin{proof}
Let $\varrho\colon Y\times G\to P$ be the quotient map of the $R(\varpi)$-action.  Then $\pi$ is well-defined and satisfies $\pi\circ \varrho=\varpi\circ\pr_1$, from which it follows that $\pi$ is smooth.  Moreover, since $\varpi$ and $\pr_1$ are subductions, it also follows that $\pi$ is a subduction. Since the action of $G$ commutes with the action of $R(\varpi)$ on $Y\times G$, the $G$-action descends via $\varrho$ to an action on $P$ whose orbits are exactly the fibers of $\pi$. 

The right $G$-action on the trivial principal $G$-bundle $\pr_1\colon Y\times G\to Y$ has action map
\[a\colon (Y\times G)\times G \to (Y\times G)\times_Y (Y\times G)\colon((y,g),h)\mapsto((y,g),(y,g+h)),\]
which is a diffeomorphism.  Moreover, $R(\varpi)$ acts on the domain of $a$ (it acts trivially on the second copy of $G$) and diagonally on the codomain, and $a$ is $R(\varpi)$-equivariant with respect to these actions.   Thus $a$ descends to a diffeomorphism $$[a]\colon((Y\times G)\times G)/R(\varpi)\cong P\times G\to((Y\times G)\times_Y (Y\times G))/R(\varpi),$$ where we identify $[(y,g),h]\in((Y\times G)\times G)/R(\varpi)$ with $([y,g],h)\in P\times G$.  

We claim that  $((Y\times G)\times_Y (Y\times G))/R(\varpi)$ is diffeomorphic to $P\times_X P$.  The smooth map $((y,g_1),(y,g_2))\mapsto([y,g_1],[y,g_2])$ is $R(\varpi)$-invariant, and so descends to a smooth map $\zeta\colon ((Y\times G)\times_Y(Y\times G))/R(\varpi)\to P\times_X P$.  On the other hand, consider the smooth map $\tilde{\xi}\colon(Y\times G)\ftimes{\varpi\circ\pr_1}{\varpi\circ\pr_1}(Y\times G)\to (Y\times G)\times_Y(Y\times G)$ defined by sending $((y_1,g_1),(y_2,g_2))$ to $((y_1,g_1),(y_1,\varphi(y_1,y_2)+g_2))$ where $\varphi$ is the $1$-cocycle $Y\times_X Y\to G$ corresponding to the $R(\varpi)$-action (see Example~\ref{x:gpd action}).  $R(\varpi)\times R(\varpi)$ has a natural action on the domain, while $R(\varpi)$ acts diagonally on the codomain. Since $\varphi$ is a $1$-cocycle, $\tilde{\xi}$ is $R(\varpi)$-equivariant with respect to the first copy of $R(\varpi)$ acting on the domain. Thus $\widetilde{\xi}$ descends to a smooth map $\overline{\xi}\colon P\ftimes{\pi}{\varpi\circ\pr_1}(Y\times G)\to((Y\times G)\times_Y(Y\times G))/R(\varpi)$, and the action of the second copy of $R(\varpi)$ descends to an action on $P\ftimes{\pi}{\varphi\circ\pr_1}(Y\times G)$.  Again, since $\varphi$ is a $1$-cocycle, $\overline{\xi}$ is invariant with respect to the descended action of $R(\varpi)$ on its domain.  Thus $\overline{\xi}$ descends to a smooth map $\xi\colon P\times_X P\to ((Y\times G)\times_Y(Y\times G))/R(\varpi)$. Since $\zeta$ and $\xi$ are inverses of each other, this proves our claim.

The composition $\zeta\circ[a]\colon P\times G\to P\times_X P$ sending $([y,g],h)$ to $([y,g],[y,g+h])$ is a diffeomorphism.  As this is the action map of the $G$-action on $P$, we conclude that $\pi\colon P\to X$ is a principal $G$-bundle.

For the last statement, let $R(\varpi)$ act on $Y\times G$ in two ways, commuting with the $G$-action, with an ($R(\varpi)\text{-}G$)-equivariant map $\Phi\colon Y\times G\to Y\times G$.  An argument similar to that in the proof of Lemma~\ref{l:groupoidrep} and Remark~\ref{r:groupoidrep} shows that there is a smooth $\alpha\colon Y\to G$ such that $\Phi(y,g)=(y,g-\alpha(y)),$ and $\Phi$ must be a diffeomorphism.  Thus, $\Phi$ descends to a diffeomorphism $[y,g]\mapsto[y,g-\alpha(y)]$ between the respective quotients by $R(\varpi)$.
\end{proof}

Finally, we prove the main result of the paper.

\begin{theorem}\labell{t:main}
Let $X$ be a diffeological space and $G$ an abelian diffeological group.  There is a natural bijection between the isomorphism classes of principal $G$-bundles on $X$ and the group of \v{C}ech cohomology classes $\check{H}^1(X,G)$.
\end{theorem}

\begin{proof}
Fix a principal bundle $\pi\colon P\to X$ with plotwise local trivialization $\tau\colon\calN(\calQ)\to P$.  Let $\tau'\colon\calN(\calQ)\to P$ be another plotwise local trivialization and $\alpha\colon\calN(\calQ)\to G$ the map $d\circ(\tau,\tau')$, where $d$ is the division map.  For any $(x,y)\in\calN(\calQ)_1$,
\[\tau'(x)\cdot (c(\tau,P)(x,y)+\partial \alpha(x,y)) = \tau'(y),\]
and so $c(\tau',P)=c(\tau,P)+\partial \alpha$.  Therefore, given a generating family $\calQ$ admitting a plotwise local trivialization of $P$, there is an associated class $[c(\tau,P)]$ independent of the plotwise local trivialization in $\check{H}^1(\calQ,G)$.

Let $\calR$ be a refinement of $\calQ$ with induced map of nebulae $g\colon\Cn(\calR)\to \Cn(\calQ)$.  By Claim~\ref{i:ev comm} of Lemma~\ref{l:refinement}, $\ev\circ g = \ev'$, and given a plotwise local trivialization $\tau$ with respect to $\calQ$ we obtain a local trivialization $\tau\circ g$ with respect to $\calR$. By \eqref{eq:bundletococycle}, the map $g^1\colon\check{C}^1(\calQ,G)\to \check{C}^1(\calR,G)\colon f\mapsto f\circ(g,g)$ as in the proof of Lemma~\ref{l:refinementtocoho} satisfies $g^1(c(\tau,P)) = c(\tau\circ g,P)$, and so the induced map on cohomology $\check{g}\colon\check{H}^1(\calQ,G)\to \check{H}^1(\calR,G)$ sends $[c(\tau,P)]$ to $[c(\tau\circ g,P)]$.  Recall that $\check{H}^1(X,G)$ is defined to be the colimit over all induced morphisms $\check{g}$ coming from all generating families (see Definition~\ref{d:cech2}). It follows that the class $[c(\tau,P)]$ induces a class $c[P]\in \check{H}^1(X,G)$ that is independent of generating families and plotwise local trivializations.

Let $\pi'\colon P'\to X$ be a principal $G$-bundle and $\varphi\colon P\to P'$ an isomorphism.  Then $\tau':=\varphi\circ\tau$ is a local trivialization for $P'$, and $c(\tau,P)=c(\tau',P')$.  Therefore, the assignment $P\mapsto c[P]$ descends to a well-defined function $\Psi$ from isomorphism classes of principal $G$-bundles over $X$ to $\check{H}^1(X,G)$.  

Fix a class $\omega\in\check{H}^1(X,G)$.  Let $\calQ$ be a generating family and $f\in\check{C}^1(\calQ,G)$ a cocycle so that $[f]$ is a representative of $\omega$.  By Lemma~\ref{l:groupoidrep}, $f$ induces an action of the groupoid $R(\ev\colon\calN(\calQ)\to X)$ on $\calN(\calQ)\times G$.  By Lemma~\ref{l:groupoidtopb}, the quotient by this action is a principal bundle $\pi\colon P:=(\calN(\calQ)\times G)/R(\ev)\to X$.  Let $\tau=\varrho\circ i$ where $i\colon\calN(\calQ)\to\calN(\calQ)\times G$ is the inclusion $y\mapsto (y,0_G)$, and $\varrho$ is the quotient map of the $R(\ev)$-action.  Since $\pi\circ\tau=\ev$, it follows that $\tau$ is a plotwise local trivialization of $P$ and by Lemma~\ref{l:pbtogroupoid}, $c(\tau,P)=f$.  It follows from above that $c[P]$ is well-defined and equal to $\omega$.  Thus $\Psi$ is surjective.{}

Let $\pi\colon P\to X$ and $\varpi\colon P'\to X$ be principal $G$-bundles such that $c[P]=c[P']$.  By Lemma~\ref{l:alwayscrs}, there exists a generating family $\calQ$ admitting plotwise local trivializations $\tau\colon\Cn(\calQ)\to P$ and $\tau'\colon\Cn(\calQ)\to P'$ such that $[c(\tau,P)]=[c(\tau',P')]$; that is, there is some $0$-cochain $\alpha$ such that $c(\tau',P')-c(\tau,P)=\del\alpha$.  By Lemma~\ref{l:groupoidrep} and Remark~\ref{r:groupoidrep}, there is an ($R(\ev)\text{-}G$)-equivariant diffeomorphism from $\calN(\calQ)\times G$ to itself sending $(x,g)$ to $(x,g-\alpha)$ that intertwines the two actions on $\calN(\calQ)\times G$ induced by $c(\tau,P)$ and $c(\tau',P')$.  By Lemmas~\ref{l:pbtogroupoid} and \ref{l:groupoidtopb}, $P$ and $P'$ are isomorphic to quotients of the actions of $R(\ev\colon\calN(\calQ)\to X)$ on $\Cn(\calQ)\times G$ induced by $c(\tau,P)$ and $c(\tau',P')$, respectively, and the $(R(\ev)\text{-}G$)-equivariant map from $\calN(\calQ)\times G$ to itself descends to an isomorphism between $P$ and $P'$.  Thus $\Psi$ is bijective.

Let $\varphi\colon Y\to X$ be a smooth map and $\pi\colon P\to X$ a principal $G$-bundle.  By Proposition~\ref{p:functorial} (and its proof), there is a plotwise local trivialization $\tau\colon\calN(\calR)\to P$, a generating family $\calQ$ of $Y$, and a smooth map $\psi\colon\calN(\calQ)\to\calN(\calR)$ such that $\ev\circ\psi=\varphi\circ\ev$ and $\varphi^*c[P]$ is represented by $\check{\psi}([c(\tau,P)])$.   Define $\tau':=(\ev,\tau\circ\psi)\colon\calN(\calQ)\to\varphi^*P$.  Then $\tau'$ is a plotwise local trivialization of $\varphi^*P$, and it is straightforward to check that $c(\tau',\varphi^*P)=\psi^\sharp(c(\tau,P))$.  Naturality of $\Psi$ follows.
\end{proof}
 
\snote{Added this Remark.  Decided to go for ``completeness'' without actually reproving everything, so feel free to reduce its length as you see fit.}
\begin{remark}
As in the case of smooth manifolds (see, for instance, \cite[Remark 4.54]{RaeburnWilliams}), Theorem~\ref{t:main} can be generalized to the case of a non-abelian diffeological group $G$ with a few adjustments.  First, for $\calQ$ a generating family of a diffeological space $X$, we call a map $\varphi:\Cn(\calQ)_1\to G$ a {\em generalized cocycle} if, for any $x$, $y$, and $z$ in the same fiber of $\ev_2$,
\[\varphi(x,y)\cdot \varphi(y,z)=\varphi(x,z)\]
for ``$\cdot$'' multiplication in $G$.  As Construction~\ref{cons:cocyclefrombundle} does not require $G$ to be abelian, we may define the generalized cocycle $c(\tau, P):\Cn(\calQ)_1\to G$ for any principal bundle $P$ trivialized by $\tau:\Cn(\calQ)\to P$.

One may similarly generalize coboundaries: say that two generalized cocycles $\varphi,\psi:\Cn(\calQ)_1\to G$ are equivalent when there exists $\alpha:\Cn(\calQ)\to G$ so that, for $x$ and $y$ in the same fiber of $\ev_1$, we have that 
\[\varphi(x,y)\cdot \alpha(y)=\alpha(x)\cdot \psi(x,y).\] 
Thus, for each generating family $\calQ$ of $X$, we get a set of equivalence classes of cocycles modulo this equivalence relation.  Unlike the abelian case, this set of classes does inherit a group structure for non-abelian $G$.  However, one may still generalize the argument proving Theorem~\ref{t:main} to a natural bijection between the set of isomorphism classes of principal $G$-bundles over any diffeological space $X$ and the colimit of the sets of equivalence classes of generalized cocycles described above.
\end{remark}

It turns out that we do not need to use the colimit $\check{H}(X;G)$ to compute the \v{C}ech cohomology of $X$, but instead a suitable generating family.

\begin{corollary}\labell{c:contr gen fam}
Let $X$ be a diffeological space with generating family $\calQ$ whose plots have contractible domains.  Then $\check{H}^1(X,G)$ is in bijection with actions of $R(\ev\colon\calN(\calQ)\to X)$ on $\calN(\calQ)\times G$ that commute with right multiplication by $G$, up to $(R(\ev)\text{-}G)$-equivariance.
\end{corollary}
{}
\begin{proof}
Fix a principal $G$-bundle $\pi\colon P\to X$.  Since the domain of each plot of $\calQ$ is contractible, there is a section $\calN(\calQ)\to\ev^*P$ which, when composed with the projection $\ev^*P\to P$, yields a plotwise local trivialization $\tau$ of $P$.  Such a $\tau$ can be obtained for any principal $G$-bundle $P$, from which it follows that isomorphism classes are in bijection with $\check{H}^1(\calQ,G)$.  The result follows from Lemma~\ref{l:groupoidrep}.
\end{proof}

\begin{remark}\labell{r:contr gen fam}
The bijection in Corollary~\ref{c:contr gen fam} is not natural.  However, given a smooth map $\varphi\colon X\to Y$ and generating families $\calQ$ of $X$ and $\calR$ of $Y$, both of which have plots with contractible domains, there is a refinement $\calQ'$ of $\calQ$ whose plots have contractible domains and for which $\varphi$ induces a map $\psi\colon\calN(\calQ')\to\calN(\calR)$ so that $\ev\circ\psi=\varphi\circ\ev$.  The map $\psi$ induces a smooth map $\psi^1\colon R(\ev\colon\calN(\calQ')\to X)\to R(\ev\colon\calN(\calR)\to Y)$, from which it follows that $\psi\times\id_G$ intertwines the two groupoid actions, and is $G$-equivariant.
\end{remark}

\section{Applications \& Examples}\labell{s:apps}

In this section, we first connect the diffeological theory of $G$-valued \v{C}ech cohomology to the classical theory found in the literature on manifolds; see Example~\ref{x:Lie}.  Next, we let $G$ be a diffeological group that is not Lie, but is of prime importance in the study of diffeology: the irrational torus (see \cite[Introduction]{PIZ}, along with content throughout the book.)  We reproduce one of the results of \cite{PIZ}, but using methods developed in this paper, in Example~\ref{x:irrational torus}.  By connecting \v{C}ech cohomology with group cohomology via crossed homomorphisms, we reproduce yet another result regarding the irrational torus; see Subsection~\ref{ss:gp cohom}.  We show that $\check{H}^k(X;\RR)=0$ for $k>0$ when $X$ is the orbit space of a proper Lie groupoid, such as a (diffeological) orbifold, in Subsection~\ref{ss:density}.  Consequently, these spaces have only trivial $\RR$-bundles.  Finally, we reproduce a classical result from sheaf theory that \v{C}ech cohomology forms the right derived functor of the global section functor; see Proposition~\ref{p:derived functor}.

\subsection{Examples}\labell{ss:examples}

\begin{example}[Smooth Manifolds]\labell{x:Lie}
Let $X$ be a smooth manifold and $G$ an abelian Lie group.  Then $X$ admits open covers as generating families, from which it follows that $\check{H}^1(X,G)$ is equal to the standard $G$-valued \v{C}ech cohomology of degree $1$ over $X$ found in the literature.  Indeed, if $\calQ=\{U_\alpha\}$ is such an open cover whose elements are identified with Euclidean open sets, then $\calN(\calQ)_0=\coprod U_\alpha$ and $\calN(\calQ)_1=\coprod U_\alpha\cap U_\beta$, from which it follows that the two cohomologies are the same.  In light of the main theorem of this paper, this is not unexpected, as principal $G$-bundles over $X$ in the diffeological sense are the same as those in the standard sense.  Moreover, if $G$ is any diffeologically discrete abelian group (that is, its diffeology is made up of constant plots), then we obtain the standard \v{C}ech cohomology with values in $G$ as defined by Spanier \cite[page 327]{spanier}, since smooth functions to such a $G$ are locally constant.
\end{example}

\begin{example}[Irrational Tori]\labell{x:irrational torus}
Let $G=\ZZ+\alpha\ZZ\subset\RR$ where $\alpha$ is an irrational number.  This is an abelian diffeologically discrete subgroup of $\RR$ with respect to addition.  The \emph{irrational torus} is the quotient $T_\alpha:=\RR/G$.  It can also be represented by the quotient $\TT^2/\Gamma$ where $\Gamma$ is the diffeological subgroup of $\TT^2$ given by $$\Gamma:=\{[t,\alpha t]\in\TT^2:=\RR^2/\ZZ^2\mid t\in\RR\};$$ see \cite[Article 8.38]{PIZ} for a proof.  It is already known that the quotient map $\pi\colon\TT^2\to T_\alpha$ is a non-trivial principal $\RR$-bundle \cite[Articles 8.38, 8.39]{PIZ}.  We can see this using the tools developed in this paper as follows.  

Let $\varrho\colon\RR\to \RR/G=:T_\alpha$ be the quotient map.  Then $\calQ:=\{\varrho\}$ is a generating family for $T_\alpha$, and $\tau\colon\RR\to\TT^2\colon x\mapsto[0,x]$ is a plotwise local trivialization of $\pi$.  It follows that $c(\tau,\TT^2)(x,x+m+n\alpha)=n\alpha$ for any $m+n\alpha\in G$. Suppose that $\beta\colon\RR\to\RR$ is a $0$-cochain so that $\del\beta(x,x+m+n\alpha)=c(\tau,\TT^2)(x,x+m+n\alpha)$.  Then 
\begin{equation}\labell{e:irrational torus}
\beta(x+m+n\alpha)=n\alpha+\beta(x)
\end{equation}
from which it follows that $\beta$ is independent of $m$, and hence factors through the circle $\mathbb{S}^1:=\RR/\ZZ$.  Thus $\beta$ must have compact image, which is absurd given \eqref{e:irrational torus}.  Thus, $\beta$ does not exist, and $c(\tau,\TT^2)$ represents a non-trivial cohomology class.  The moral is that real-line bundles over diffeological spaces are not necessarily trivial, and can have interesting cohomology.
\end{example}

\subsection{Relation to Group Cohomology}\labell{ss:gp cohom}
This example is motivated by (and is a generalization of) \cite[Article 8.39]{PIZ}.  Suppose $X$ is a diffeological space, $\calQ$ a generating family for $X$, and $G$ an abelian diffeological group.  Suppose that $\calH\toto\calN(\calQ)$ is a diffeological groupoid with source and target $s$ and $t$, resp., such that $(t,s)\colon\calH\to\calN(\calQ)\times\calN(\calQ)$ has image $R(\ev)$.  Given any coycle $f\in\check{C}^1(\calQ,G)$, there is a diffeological groupoid action of $\calH$ on $\calN(\calQ)\times G$ that factors through the action of $R(\ev)$ on $\calN(\calQ)\times G$ induced by $f$, as in Example~\ref{x:gpd action}.

Suppose that $\calH$ is the translation groupoid $\calN(\calQ)\rtimes K$ for a diffeological group $K$; that is, there is some smooth action of $K$ on $\calN(\calQ)$ such that $\ev(y_1)=\ev(y_2)$ if and only if there exists $k\in K$ with $y_2=y_1\cdot k$.  Then $\CIN(\calN(\calQ),G)$ is a $K$-module with $(h\cdot k)(y)=h(y\cdot k)$, and a cocycle $f\in\check{C}^1(\calQ,G)$ induces a \emph{crossed homomorphism} $\kappa_f\colon K\to\CIN(\calN(\calQ),G)$ given by $$\kappa_f(k)(y)=f(y,y\cdot k);$$ this means that for all $k,k'\in K$, $$\kappa_f(kk')(y)=\kappa_f(k)(y)+(\kappa_f(k')\cdot k)(y).$$  Moreover, $\kappa_f$ is smooth with respect to the functional diffeology on $\CIN(\calN(\calQ),G)$ (see Example~\ref{x:functions to gp}), the correspondence $f\mapsto \kappa_f$ is injective since $(t,s)(\calN(\calQ)\rtimes K)=R(\ev)$, and also a homomorphism of abelian groups $\check{C}^1(\calQ,G)\to\Hom(K,\CIN(\calN(\calQ),G))$.  Furthermore, if $f=\del\alpha$ for some $0$-cochain $\alpha\in\check{C}^0(\calQ,G)=\CIN(\calN(\calQ),G)$, then $\kappa_{\del\alpha}(k)=\alpha\cdot k-\alpha$; that is, $\kappa_{\del\alpha}$ is a so-called \emph{principal crossed homomorphism}.  It follows that we obtain a map $\kappa\colon\check{H}^1(\calQ,G)\to H^1(K,\CIN(\calN(\calQ),G))$, where the latter cohomology group is the first group cohomology of the $K$-module $\CIN(\calN(\calQ),G)$, equal to crossed homomorphisms modulo principal crossed homomorphisms.

Now suppose further that $(t,s)\colon\calN(\calQ)\rtimes K\to\calN(\calQ)\times\calN(\calQ)$ is an induction; that is, $\calN(\calQ)\rtimes K$ is diffeomorphic to $R(\ev)$, which implies that the action of $K$ is free.  We obtain a homomorphism sending $\beta\in\CIN(K,\CIN(\calN(\calQ),G))$ to $f_\beta\in\check{C}^1(\calQ,G)$ defined by $$f_\beta(y_1,y_2)=\beta(k(y_1,y_2))(y_1);$$ here $k(y_1,y_2)$ is the unique $k\in K$ such that $y_2=y_1\cdot k$, which depends on $(y_1,y_2)$ smoothly as $(t,s)$ is an induction.  If $\beta$ is a crossed homomorphism, then $f_\beta$ is a cocycle.  If $\beta$ is a principal cross homomorphism $k\mapsto \alpha\cdot k-\alpha$ for some $\alpha\in\CIN(\calN(\calQ),G)=\check{C}^0(\calQ,G)$, then $f_\beta=\del\alpha$.  The correspondences described here and in the previous paragraph are inverses of each other, establishing a bijection between $1$-cocycles and cross homomorphisms which descends to an isomorphism $$\check{H}^1(\calQ,G)\cong H^1(K,\CIN(\calN(\calQ),G)).$$

If we chose $\calQ$ so that its plots have contractible domains, then by Corollary~\ref{c:contr gen fam}, $\check{H}^1(X,G)\cong H^1(K,\CIN(\calN(\calQ),G))$. For an explicit example of this phenomenon, consider the irrational torus of Example~\ref{x:irrational torus}, where $K=\ZZ+\alpha\ZZ\subset\RR$ and $\alpha\in\RR\smallsetminus\QQ$.  If $\calQ=\{\varrho\colon\RR\to T_\alpha\}$ where $\varrho$ is the quotient map by the $K$-action on $\RR$, then $R(\ev)$ is isomorphic to $\RR\rtimes K$, and in conjunction with Theorem~\ref{t:main}, we reproduce the result in \cite[Article 8.39]{PIZ}: principal $\RR$-bundles over $T_\alpha$ are in bijection with $H^1(K,\CIN(\calN(\calQ),G))$.

\subsection{Orbit Spaces of Proper Lie Groupoids, and Orbifolds}\labell{ss:density}
Consider a proper Lie groupoid $\calG_1\toto\calG_0$ whose orbit space $X=\calG_0/\calG_1$ is equipped with the quotient diffeology.  Let $\calV$ be an open cover of $\calG_0$ in which each element is identified diffeomorphically with an open subset of some $\RR^n$, and let $\calU$ be the inclusion maps of elements of $\calV$ composed with the quotient map $\calG_0\to X$. Then $\calU$ is a generating family for the diffeology on $X$. Let $i\colon\calN(\calU)=\coprod\calV\to\calG_0$ be the inclusion map into $\calG_0$ when restricted to each element of $\calV$. The pullback groupoid $i^*\calG$ has base space $\calN(\calU)_0$, arrow space $i^*\calG_1:=\{(u_1,u_2,g)\mid s(g)=i(u_1),\,t(g)=i(u_2)\}$, and source $s$ and target $t$ the first and second projection maps, resp., making $i^*\calG$ into a proper Lie groupoid.

By \cite[Proposition 10.6]{CrainicMestre}, $i^*\calG$ admits a proper normalized Haar density $\rho$; that is, a smooth family of normalized densities on the fibers of $s$ for which $s|_{t^{-1}(\supp(\rho))}\colon t^{-1}(\supp(\rho))\to\calN(\calU)_0$ is proper and which is invariant under right multiplication by elements of $(i^*\calG)_1$ (see \cite[Section 10]{CrainicMestre} for details): $$\rho^u(f):=\int_{s^{-1}(u)}f\,d\mu^u; \quad f\in\CIN(s^{-1}(u)),~u\in\calN(\calU)_0.$$  Define $\delta$ to be the smooth map
$$\delta(u)(h):=\rho^u(t^*h|_{s^{-1}(u)}); \quad u\in\calN(\calU)_0,~h\in\CIN(\calO_u)$$ where $\calO_u$ is the orbit of $u\in\calN(\calU)_0$.  Since the orbits of $i^*\calG$ are exactly the orbits of $\calN(\calU)_1\toto\calN(\calU)_0$, it follows that $\delta$ is $\calN(\calU)_1$-invariant.  Also, since $\rho$ is normalized, $\delta(u)(1)=1$ for all $u$.  Finally, $\delta$ is linear in $h$.

Fix a cocycle $f\in\check{C}^k(\calU;\RR)$ with $k>0$.  Define $g\colon\calN(\calU)_{k-1}\to\RR$ by
$$g(u_0,\dots,u_{k-1}):=\delta(u_0)(f(u_0,\dots,u_{k-1},\cdot)).$$  Since $g$ is smooth, it is in $\check{C}^{k-1}(\calU;\RR)$.  It is straightforward to check that $\del g=(-1)^kf$, from which it follows that $f$ is a coboundary.  Thus, we have that $\check{H}^k(\calU,\RR)=0$ for all $k>0$.

As any generating family of $X$ can be refined to a generating family of the form $\calU$ (take, for instance, a sufficiently fine good cover of $\calG_0$ for $\calV$), we have $\check{H}^k(X,\RR)=0$ for $k>0$.  In particular, all principal $\RR$-bundles over $X$ are trivial.

An immediate consequence for orbifolds is the following.  The ``category'' of orbifolds is often taken to be the bicategory of proper \'etale Lie groupoids (with right bibundles as arrows and bi-equivariant diffeomorphisms for $2$-arrows); see, for instance, \cite{Lerman}.  In the effective case, passing to their orbit spaces yields a functor into diffeological spaces, which is injective on objects up to isomorphism; see \cite{IKZ,Watts-gpds}.  Hence, it makes sense to define orbifolds as diffeological spaces.  It follows from above that $\check{H}^k(X,\RR)=0$ for $k>0$ for any orbifold $X$ (defined as Lie groupoids or as diffeological spaces).  

Finally, we can replace $\RR$ with $\RR^n$ with very little modification to the argument above.  This yields that $\check{H}^k(X,\RR^n)=0$ for $k>0$ for an orbit space $X$ of a proper Lie groupoid.

\subsection{\v{C}ech Cohomology as a Right Derived Functor}\labell{ss:derived functor}

Following Raeburn and Williams, we show that \v{C}ech cohomology forms the right derived functor of the global sections functor, just as in the classical theory of sheaves.

\begin{proposition}\labell{p:derived functor}
	Let $X$ be a diffeological space, and let $$0\Rightarrow F\overset{\Phi}{\Rightarrow} G\overset{\Psi}{\Rightarrow} H\Rightarrow 0$$ be a short exact sequence of sheaves over $X$.  The following is a long exact sequence of abelian groups:
	$$0\to F(X)\overset{\widetilde{\Phi}}{\longrightarrow} G(X) \overset{\widetilde{\Psi}}{\longrightarrow} H(X) \to\check{H}^1(X,F)\to\check{H}^1(X,G)\to\check{H}^1(X,H)\to\dots.$$
\end{proposition}

\begin{proof}
	Let $X$ be a diffeological space with generating family $\calQ$ of $\calD_X$, and let $$0\Rightarrow F\overset{\Phi}{\Rightarrow} G\overset{\Psi}{\Rightarrow} H\Rightarrow 0$$ be a short exact sequence of sheaves over $X$.  By Lemma~\ref{l:ses} and Proposition~\ref{p:ses}, we obtain for each $k\geq 0$ an exact sequence
$$0\to\check{C}^k(\calQ,F)\overset{\Phi^k_\sharp}{\longrightarrow}\check{C}^k(\calQ,G)\overset{\Psi^k_\sharp}{\longrightarrow}\check{C}^k(\calQ,H).$$
Define $\calT_\calQ^k:=\im(\Psi^k_\sharp)$.  Since $\del\circ\Psi^k_\sharp=\Psi^{k+1}_\sharp\circ\del$ (see the proof of Proposition~\ref{p:morph cohom}), $\del$ sends $\calT_\calQ^k$ to $\calT_\calQ^{k+1}$ for each $k$.  Thus we have a short exact sequence of complexes:
	$$\xymatrix{
		 \ar@{=}[d] & \ar[d]^{\del} & \ar[d]^{\del} & \ar[d]^{\del} & \ar@{=}[d] \\
		0 \ar[r] \ar@{=}[d] & \check{C}^{k-1}(\calQ,F) \ar[r]^{~~~~\Phi^{k-1}_\sharp} \ar[d]^{\del} & \check{C}^{k-1}(\calQ,G) \ar[r]^{\quad\Psi^{k-1}_\sharp} \ar[d]^{\del} & \calT_\calQ^{k-1} \ar[r] \ar[d]^{\del} & 0 \ar@{=}[d] \\ 
		0 \ar[r] \ar@{=}[d] & \check{C}^k(\calQ,F) \ar[r]^{\Phi^k_\sharp} \ar[d]^{\del} & \check{C}^k(\calQ,G) \ar[r]^{\Psi^k_\sharp} \ar[d]^{\del} & \calT_\calQ^k \ar[r] \ar[d]^{\del} & 0 \ar@{=}[d] \\
		0 \ar[r] \ar@{=}[d] & \check{C}^{k+1}(\calQ,F) \ar[r]^{~~~~\Phi^{k+1}_\sharp} \ar[d]^{\del} & \check{C}^{k+1}(\calQ,G) \ar[r]^{\quad\Psi^{k+1}_\sharp} \ar[d]^{\del} & \calT_\calQ^{k+1} \ar[r] \ar[d]^{\del} & 0 \ar@{=}[d] \\  
		& & & & \\
	}$$
Letting $\calH_\calQ^k:=\ker(\del|_{\calT_\calQ^k})/\im(\del|_{\calT_\calQ^{k-1}})$, there is thus a long exact sequence
$$0\to\check{H}^0(\calQ,F)\to\check{H}^0(\calQ,G)\to\calH_\calQ^0\to\check{H}^1(\calQ,F)\to\check{H}^1(\calQ,G)\to\calH_\calQ^1\to\dots.$$
Let $\calR$ be a refinement of $\calQ$ with induced map $f\colon\calN(\calR)\to\calN(\calQ)$.  It follows from the definitions and the fact that $\Psi^{k+1}_\sharp\circ\del=\del\circ\Psi^k_\sharp$ that $f^\sharp$ sends $\calT^k_\calQ$ to $\calT^k_\calR$, and since $f^\sharp\circ\del=\del\circ f^\sharp$ (see the proof of Lemma~\ref{l:refinementtocoho}), we obtain a map $f^*\colon\calH^k_\calQ\to\calH^k_\calR$.  Let $\calH^k$ be the directed limit of $\calH^k_\calQ$ over refinements $\calQ$.  We obtain the diagram
$$0\to\check{H}^0(X,F)\to\check{H}^0(X,G)\to\calH^0\to\check{H}^1(X,F)\to\check{H}^1(X,G)\to\calH^1\to\dots$$ where the maps are each induced by the universal property of directed limits.  It is straightfoward using common refinements to show that this is an exact sequence.

The work above also shows that the inclusion $\calT^k_\calQ\to\check{C}^k(\calQ,H)$ induces a map $I\colon\calH^k\to\check{H}^k(X,H)$.  Fix $\mu\in\check{H}^k(X,H)$.  There exists a generating family $\calQ$ of $\calD_X$ for which $\mu$ is represented by a cocycle $\mu'\in\check{C}^k(\calQ,H)$.  Since $0\Rightarrow F\Rightarrow G\Rightarrow H\Rightarrow 0$ is a short exact sequence of sheaves, for any plot $(f_0,\dots,f_k)$ of $\calN(\calQ)_k$ and any $u\in U_{(f_0,\dots,f_k)}$ there is an open neighbourhood $V_u\subseteq U_{(f_0,\dots,f_k)}$ of $u$ such that $i_{V_u}^*\mu'\in\im(\ev_k^*\Psi)_{(f_0,\dots,f_k)|_{V_u}}$ where $i_{V_u}\colon V_u\to U_{(f_0,\dots,f_k)}$ is the inclusion.  Let $\calR$ be the refinement of $\calQ$ whose plots are restrictions of $(f_0,\dots,f_k)$ to each $V_u$, letting $(f_0,\dots,f_k)$ run over all plots of $\calN(\calQ)_k$, and let $i$ be the induced map $\calN(\calR)_k\to\calN(\calQ)_k$.  Then $i^\sharp\mu'\in\im(\ev_k^*\Psi^k_\sharp)=\calT_\calR^k$.  Since $i^\sharp$ commutes with $\del$, $i^\sharp\mu'$ is also a cocycle representing $\mu$.  It follows that $I$ is surjective.

Similarly, if $\mu\in\calH^k$ such that $I(\mu)=0$, then there exist a generating family $\calQ$ of $\calD_X$ and cocycle $\mu'\in\calT^k_\calQ$ so that $\mu'=\del\eta$ for some $\eta\in\check{C}^{k-1}(\calQ,H)$.  Constructing a similar refinement $\calR$ to $\calQ$ as above with induced map $i\colon\calN(\calR)_k\to\calN(\calQ)_k$ so that $i^*\eta\in\calT_\calR^{k-1}$, it follows that $i^\sharp\mu'=i^\sharp\del\eta=\del i^\sharp\eta$.  It now follows that $\mu=0$; that is, $I$ is injective, and so $I$ is an isomorphism from $\calH^k$ to $\check{H}^k(X,H)$.  The result now follows from Proposition~\ref{p:deg-0}.
\end{proof}

	Let $X$ be the orbit space of a proper Lie groupoid, such as an orbifold, equipped with the quotient diffeology.  Let $\Gamma$ be a diffeologically discrete subgroup of $\RR^n$, such as $\ZZ^n$ or $\ZZ+\alpha\ZZ$ where $\alpha\in\RR\smallsetminus\QQ$; recall that this means that the subset diffeology on $\Gamma$ is made up of only constant plots.  Let $T=\RR^n/\Gamma$ be the diffeological quotient group.  By Subsection~\ref{ss:density}, $\check{H}^k(X,\RR^n)=0$ for $k>0$.  Thus by Proposition~\ref{p:derived functor}, there is an isomorphism $\check{H}^1(X,T)\cong\check{H}^2(X,\Gamma)$.  By Theorem~\ref{t:main}, we obtain a bijection from $\check{H}^2(X,\Gamma)$ to isomorphism classes of principal $T$-bundles on $X$.  By Example~\ref{x:Lie}, we recover the standard result for a manifold $X$ that its principal $n$-torus bundles are classified by $\check{H}^2(X,\ZZ^n)$, which in turn is isomorphic to the second $\ZZ^n$-valued singular cohomology group of $X$.  In fact, this holds with $\ZZ^n$ replaced with any diffeologically discrete subgroup of $\RR^n$: for any such group $\Gamma$ the $\Gamma$-valued singular cohomology group $H^2(X;\Gamma)$ of a manifold $X$ is isomorphic to the \v{C}ech cohomology group $\check{H}^2(X,\Gamma)$, and hence is in bijection with isomorphism classes of principal ($T=\RR^n/\Gamma$)-bundles over $X$.  A future direction of research is to study exactly how this generalises to more general diffeological spaces $X$ and groups $G$, if at all.  In the meantime:

\begin{corollary}\labell{c:derived functor}
	Let $X$ be the orbit space of a proper Lie groupoid and $T$ the diffeological group $\RR^n/\Gamma$ where $\Gamma\subset\RR^n$ is diffeologically discrete.  Then principal $T$-bundles of $X$ are classified by $\check{H}^2(X,\Gamma)$.  In particular, if $X$ is a manifold, then $\check{H}^2(X,\Gamma)$ is isomorphic to the singular cohomology group $H^2(X;\Gamma)$.
\end{corollary}

\subsection{Comparison with the \v{C}ech Cohomology of Iglesias-Zemmour}\labell{ss:iz cohom}

In this section, we compare the \v{C}ech cohomology on diffeological spaces defined by Iglesias-Zemmour in \cite{PIZCech}.  We first review some terminology and notation from \cite[Sections 8, 10]{PIZCech}.  Fix a diffeological space $X$, and let $\calQ$ be the generating family of all \textbf{round plots}: plots whose domains are open Euclidean balls (of all dimensions).  Let $\calM$ be the corresponding \textbf{gauge monoid}: $$\calM:=\{m\in\CIN(\calN(\calQ),\calN(\calQ))\mid \ev\circ m=\ev \text{ \& $\exists$ finitely-many $p\in\calQ$ s.t. } m|_{U_p}\neq\id_{U_p}\}.$$  Equip $\calM$ with the subset diffeology induced by the standard functional diffeology.  Let $\calA$ be the sheaf of locally constant functions into a fixed abelian diffeological group $G$.  This is a domain-determined sheaf, and so the pullback of $\calA$ by a smooth map yields again the sheaf of locally constant functions to $G$.

The \v{C}ech cohomology defined in \cite[Section 10]{PIZCech} is the Hochschild cohomology of the monoid $\calM^k$ acting on the $\calM$-module of locally constant $G$-valued functions on $\calN(\calQ)$; here, elements of $\calM$ act by pullback.  The cochains are given by
	\begin{align*}
		\check{C}_{IZ}^0(\calM,\calA):=&~\calA(\calN(\calQ)), \\
		\check{C}_{IZ}^k(\calM,\calA):=&~\CIN(\calM^k,\calA(\calN(\calQ))),\quad\text{for $k>0$.}
	\end{align*}
The differential is given by
	\begin{align*}
		\delta(\alpha)(m_0)(y):=&~\alpha(m_0(y))-\alpha(y), \quad \text{for $\alpha\in\check{C}_{IZ}^0(\calM,\calA)$}, \\
		\delta(\eta)(m_0,\dots,m_k)(y):=&~\eta(m_1,\dots, m_k)(m_0(y))+\sum_{i=1}^k(-1)^i\eta(m_0,\dots,m_i\circ m_{i-1},\dots,m_k)(y)\\
								+&~(-1)^{k+1}\eta(m_0,\dots,m_{k-1})(y) 
	\end{align*}
		for $\eta\in\check{C}_{IZ}^k(\calM,\calA)$ with $k>0$, where $(m_0,\dots,m_k)\in\calM^{k+1}$ and $y\in\calN(\calQ)$.  Denote the corresponding cohomology groups by $\check{H}_{IZ}^k(X,G)$.  It follows from the definitions that $\check{H}_{IZ}^0(X,G)$ is the group of locally constant $G$-valued functions on $X$, which by Proposition~\ref{p:deg-0} is equal to $\check{H}^0(X,\calA)$.

	\begin{lemma}\labell{l:loc const sheaf}
		The global sections $\calA(X)$ are exactly the locally constant $G$-valued functions on $X$.
	\end{lemma}

	\begin{proof}
		Given a coherent family $\eta=\{\eta_p\}$, if $f_p^q$ is an arrow between plots $p$ and $q$ with connected domains, then the constant $\eta_p\in G$ is equal to the constant $\eta_q\in G$.  As the D-topology of a diffeological space is locally path-connected (see \cite{CSW,hector,laubinger}), it follows that $\eta$ induces a function $X\to G$ that is locally constant.  Conversely, any locally constant function $X\to G$ induces a global section of $\calA$.
	\end{proof}


It may also be useful to view $\calA(X)$ as the global sections of the sheaf $\CIN(\cdot, G_{\operatorname{discr}})$ on $X$, where $G_{\operatorname{discr}}$ is $G$ with the discrete diffeology.

The map $$\chi\colon\calM\times\calN(\calQ)\to \calN(\calQ)_1\colon(m,y)\mapsto(y,m(y))$$ is well-defined and smooth.  By the exponential law for the category of diffeological spaces, we identify $\CIN(\calM^k,\CIN(\calN(\calQ),G))$ with $\CIN(\calM^k\times\calN(\calQ),G)$ for each $k$.  Given $\eta\in\check{C}^1(\calQ,\calA)$, the pullback $\chi^*\eta$ is a locally constant $G$-valued function on $\calM\times\calN(\calQ)$, which becomes uniquely identified with an element of $\check{C}^1(\calM,\calA)$.

	\begin{lemma}\labell{l:del chi commute}
		The pullback map $\chi^*$ descends to a map $\check{H}^1(\calQ,\calA)\to\check{H}^1_{\IZ}(X,G)$ which we also denote by $\chi^*$.
	\end{lemma}

	\begin{proof}
		Let $\eta$ be a $1$-cocycle: for any plot $(f_0,f_1,f_2)$ of $\calN(\calQ)_2$, $$0=\del\eta(f_0,f_1,f_2)=\eta(f_1,f_2)-\eta(f_0,f_2)+\eta(f_0,f_1).$$  For any $(m_0,m_1)\in\calM^2$ and $y\in\calN(\calQ)$,
			\begin{align*}
				\delta(\chi^*\eta)(m_0,m_1)(y)=&~\chi^*\eta(m_1)(m_0(y))-\chi^*\eta(m_1\circ m_0)(y)+\chi^*\eta(m_0)(y) \\
				=&~\eta(m_0(y),m_1\circ m_0(y))-\eta(y,m_1\circ m_0(y))+\eta(y,m_0(y))\\
				=&~0
			\end{align*}
		where the last line follows from the cocycle condition.  Thus $\chi^*\eta$ is a $1$-cocycle in $\check{C}_\IZ^1(X,G)$.

		If $\eta=\del\alpha$ for some $\alpha\in\check{C}^0(\calQ,\calA)$, then for any $m\in\calM$ and $y\in\calN(\calQ)$,
			$$\chi^*\eta(m,y)=\eta(y,m(y))=\alpha(m(y))-\alpha(y).$$  Here, $\alpha$ is a locally constant $G$-valued function on $\calN(\calF)_0$, and so also is an element of $\check{C}_\IZ^1(X,G)$.  Applying $\delta$ to it: $$\delta\alpha(m,y)=\alpha(m(y))-\alpha(y)=\chi^*\eta(m,y).$$
The result follows.
	\end{proof}

We wish to show that $\chi^*$ is injective on cohomology; this is a corollary of the following lemma.  Recall that $\calN(\calQ)_1$ is the arrow space of the groupoid $R(\ev)$ (see Example~\ref{x:gpd action} for a definition).

	\begin{lemma}\labell{l:almost surj chi}
		Given $(y',y)\in\calN(\calQ)_1$ there exist $m,m'\in\calM$ and $x\in\calN(\calQ)$ so that $(y',y)=\chi(m',x)^{-1}\chi(m,x)$.
	\end{lemma}

	\begin{proof}
		Let $p,p'\in\calQ$ so that $y\in U_p$ and $y'\in U_{p'}$, and let $x=\ev(y)=\ev(y')$.  Denote by $\widehat{x}\colon\RR^0\to X$ the inclusion of $x$ into $X$; this is in $\calQ$.  Define $m\in\calM$ to be the smooth map sending $x\in U_{\widehat{x}}\subset\calN(\calQ)$ to $y\in U_p\subset\calN(\calQ)$ and equal to the identity elsewhere, and $m'\in\calM$ to be the smooth map sending $x$ to $y'\in U_{p'}\subset\calN(\calQ)$ and equal to the identity elsewhere.  Then $\chi(m,x)=(x,y)$ and $\chi(m',x)=(x,y')$, and the result follows.
	\end{proof}

	\begin{corollary}\labell{c:almost surj chi}
		The map $\chi^*\colon\check{H}^1(\calQ,\calA)\to\check{H}_\IZ^1(X,G)$ is injective.
	\end{corollary}

	\begin{proof}
		Suppose $\eta$ is a $1$-cocycle in $\check{C}^1(\calQ,\calA)$ so that $\chi^*\eta=\delta\alpha$ for some $\alpha\in\check{C}^0_\IZ(X,G)$.  Fix $(y',y)\in\calN(\calQ)$.  By Lemma~\ref{l:almost surj chi}, there exist $m,m'\in\calM$ and $x\in\calN(\calQ)$ such that $(y',y)=\chi(m',x)^{-1}\chi(m,x)$.  By the cocycle condition, $\eta(y',y)=-\eta(\chi(m',x))+\eta(\chi(m,x))$.  But $\delta(\alpha)(m',x)=\alpha(m'(x))-\alpha(x)$ and $\delta(\alpha)(m,x)=\alpha(m(x))-\alpha(x)$, hence $\eta(y',y)=\alpha(y)-\alpha(y')=\del\alpha(y',y)$, where we identified $\check{C}^1_\IZ(X,G)$ with $\check{C}^1(\calQ,\calA)$. The result follows.
	\end{proof}

	There is no reason to believe that the map $\chi^*$ is a surjection, however.  Indeed, let $\eta\in\check{C}^1(\calQ,\calA)$ be a $1$-cocycle.  Then for any $m\in\calM$ and $y\in\calN(\calQ)$, we have $\chi^*\eta(m,y)=\eta(y,m(y))$.  In particular, any plot $t\mapsto m_t$ of $\calM$ must keep $m_t(y)$ in the same plot domain, and hence $\eta(y,m_t(y))$ must be constant.  However, given $\mu\in\check{C}^1_\IZ(X,G)$, a plot $t\mapsto m_t$ of $\calM$ may vary the locally constant map $\mu(m_t)$ as $t$ varies: for fixed $y\in\calN(\calQ)$, we obtain a path $\mu(m_t)(y)$ in $G$.  Such $\mu$ cannot be in the image of $\chi^*$ unless $G$ has the discrete diffeology, and it is not clear that adding on a coboundary would remedy this issue.  (One may also suggest changing $\check{C}^1(\calQ,\calA)$ to $\check{C}^1(\calQ,G):=\check{C}^1(\calQ,\CIN(\cdot,G))$ in order to obtain an isomorphism on cohomologies, but $\chi^*$ is no longer well-defined on cochains in this case: given $\eta\in\check{C}^1(\calQ,G)$, the pullback $\chi^*\eta$ is not a map into $\calA$.)

	\begin{lemma}\labell{l:round plot}
		The two cohomlogy groups $\check{H}^1(\calQ,\calA)$ and $\check{H}^1(X,\calA)$ are isomorphic, and hence $\chi^*$ induces an injection $\check{H}^1(X,\calA)\to\check{H}^1_\IZ(X,G)$.
	\end{lemma}

	\begin{proof}
		The first clause follows from the observation that $\calA(\calN(\calQ)_1)=\CIN(\calN(\calQ)_1,G_{\operatorname{discr}})$ and Corollary~\ref{c:contr gen fam}.  The second clause follows from Corollary~\ref{c:almost surj chi}.
	\end{proof}

	For higher order cohomology, we have a map $\chi_k\colon\calM^k\times\calN(\calQ)\to\calN(\calQ)_k$ defined by $$\chi_k(m_1,\dots,m_k,y):=(y,m_1(y),\dots,m_k(y)).$$  It is straightforward to show that $\chi_k\circ\del=\delta\circ\chi_{k-1}$, generalizing the proof of Lemma~\ref{l:del chi commute}, and hence $\chi^*$ descends to a map on cohomology: $\chi^*\colon\check{H}^k(\calQ,\calA)\to\check{H}_\IZ^k(X,G)$.

\end{document}